\documentclass[reqno]{amsart}
\usepackage{mathpazo}
\usepackage[colorlinks=true, pdfstartview=FitV, linkcolor=blue,
            citecolor=blue, urlcolor=blue]{hyperref} 
\usepackage[usenames, dvipsnames]{color}
\usepackage{booktabs}
\usepackage{textcomp}
\numberwithin{equation}{section}
\usepackage[sort]{cite}

\usepackage{fancyhdr}

\fancypagestyle{firststyle}
{
   \fancyhf{}
   
   \cfoot{}
   \rfoot{\footnotesize \vspace{1em} \today}
   \lfoot{\footnotesize \vspace{1em} Preprint accepted for publication by EJDE} 
}
\newtheorem{thm}{Theorem}[section]
\newtheorem{lem}{Lemma}[section]
\newtheorem{defi}{Definition}[section]

\newtheorem{rem}{Remark}[section]
\newtheorem{cor}{Corollary}[section]
\def\bt{\begin{thm}}
\def\et{\end{thm}}
\def\bl{\begin{lem}}
\def\el{\end{lem}}
\def\bd{\begin{defi}}
\def\ed{\end{defi}}
\def\bc{\begin{cor}}
\def\ec{\end{cor}}
\def\bp{\begin{proof}}
\def\ep{\end{proof}}
\def\br{\begin{rem}}
\def\er{\end{rem}}
\def\bi{\begin{itemize}}
\def\ei{\end{itemize}}
\def\Lip{{\rm Lip}}

\def\cH{\mathcal H}
\def\Forall{\text{ } \forall \:}
\def\d{\mathrm{d}}
\def\v{v}
\def\R{\mathbb{R}}

\def\be{\begin{equation}}
\def\ee{\end{equation}}
\def\bes{\begin{equation*}}
\def\ees{\end{equation*}}
\def\bea{\begin{equation} \begin{aligned}}
\def\eea{\end{aligned} \end{equation}}
\def\beas{\begin{equation*} \begin{aligned}}
\def\eeas{\end{aligned} \end{equation*}}

\begin{document}
\title[Galerkin approximations of nonlinear optimal control problems]
{Galerkin approximations of nonlinear optimal control problems in Hilbert spaces}

\author[Micka\"el D. Chekroun]{Micka\"el D. Chekroun}
\address{Department of Atmospheric \& Oceanic Sciences, University of California, Los Angeles, CA 90095-1565, USA}
\email{mchekroun@atmos.ucla.edu}

\author[Axel Kr\"oner]{Axel Kr\"oner}
\address{INRIA and CMAP, \'Ecole Polytechnique, CNRS, Universit\'e Paris Saclay, 91128 Palaiseau, France}
\email{axel.kroener@inria.fr}

\author[Honghu Liu]{Honghu Liu}
\address{Department of Mathematics, Virginia Polytechnic Institute and State University, Blacksburg, Virginia 24061, USA}
\email{hhliu@vt.edu}

\begin{abstract}
Nonlinear optimal control problems in Hilbert spaces are considered for which we derive 
approximation theorems for Galerkin approximations. Approximation theorems are available in the literature. The originality of our approach
relies on the identification of a set of natural assumptions that allows us to deal with a broad class of nonlinear evolution equations and cost functionals for which we derive convergence of the value functions associated with the optimal control problem of the Galerkin approximations. This convergence result holds for a broad class of nonlinear control strategies as well. In particular, we show that the framework applies to the optimal control of semilinear heat equations posed on a general compact manifold without boundary.   The framework is then shown to apply to geoengineering and mitigation of greenhouse gas emissions formulated here in terms of optimal control of energy balance climate models posed on the sphere $\mathbb{S}^2$.

\bigskip
\noindent{\bf \keywordsname}. Nonlinear optimal control problems;  Galerkin approximations; Greenhouse gas emissions; Energy balance models; Trotter-Kato approximations

\end{abstract}

%\subjclass[2010]{35Q86, 35Q93, 35K58, 49J15, 49J20, 86-08}
%\keywords{Nonlinear optimal control problems;  Galerkin approximations; 
%\hfill\break\indent Greenhouse gas emissions; Energy balance models; Trotter-Kato approximations}

\maketitle
\thispagestyle{firststyle}

\tableofcontents

\section{Introduction}

Optimal control problems of infinite dimensional systems play an important role in a broad range of applications in engineering and various scientific disciplines \cite{Lions71,Fattorini99,Fur00,bensoussan2007representation,HPUU09,Tro10}. Various methods for solving numerically the related optimization problems are available; see e.g.~\cite{medjo2008optimal,HPUU09}. 
The case of linear evolution equations has benefited from a long tradition, and an abundant literature exists about finite element techniques or Galerkin methods for the design of approximate optimal controls; see e.g.~\cite{Mcknight_al73,Gibson79,lasiecka1980unified,malanowski1982convergence,knowles1982finite,Banks_al84,lasiecka1987regulator,Alt_al89,Lasiecka_al00}.   The case of Galerkin approximations of optimal control problems for nonlinear evolutions seems to have been much less addressed. Semidiscrete Ritz-Galerkin approximations of nonlinear parabolic boundary control problems have been considered for which convergence of the approximate controls have been obtained; see \cite{MR1297998,MR1264014}. We refer also to \cite{meidner2007adaptive,Neitzel2012} for error estimates concerned with space-time finite element approximations of the state and control to optimal control problems governed by semilinear parabolic equations, and to \cite{deckelnick2004semidiscretization} for finite element approximations of optimal control problems associated with the Navier-Stokes equations.

In this article, we study Galerkin approximations for (possibly non-quadratic) optimal control problems over a finite horizon $[0,T]$, of nonlinear evolution equations in Hilbert space. Our framework covers not only a broad class of semilinear parabolic equations but also includes systems of nonlinear delay differential equations (DDEs)  \cite{CGLW15} and allows in each case for a broad class of nonlinear control strategies. The main contribution of this article is to identify for such equations a set of easily checkable conditions in practice, from which we prove the pointwise convergence of the value functions associated with the optimal control problem of the Galerkin approximations, and for a broad class of cost functionals; see Theorem~\ref{Thm_cve_Galerkin_val}, our main result. This convergence at the level of value functions results essentially from a double uniform convergence---with respect to time and the set of admissible controllers---of the controlled Galerkin states; see Theorem \ref{Lem:uniform_in_u_conv} and Corollary \ref{Lem:uniform_conv_locally_Lip_t_v2}.

The treatment adopted here is based on the classical Trotter-Kato approximation approach from the $C_0$-semigroup theory \cite{Pazy83,Goldstein85}, which can be viewed as the functional analysis operator version of the Lax equivalence principle.\footnote{i.e., if ``consistency'' and ``stability'' are satisfied, then ``convergence'' holds, and reciprocally.} Within this approach, we generalize, in particular, the convergence results about value functions obtained in the earlier work \cite{Ferretti97} concerned with the Galerkin approximations to optimal control problems governed by linear evolution equations in Hilbert space.  Given a Hilbert state space $\cH$, denoting by $\Pi_N$ the orthogonal projector associated with the $N$-dimensional Galerkin subspace, and by $y_N(\cdot; \Pi_N x, u)$  the controlled  Galerkin state (driven by $u$) and emanating from $\Pi_N x$, a key property to ensure convergence of the value functions for such optimal control problems is the following double uniform convergence
\be \label{uniform_in_u_conv_Intro}
\lim_{N\rightarrow \infty}  \sup_{u\in \mathcal{U}_{ad}} \sup_{t \in [0, T]} \|y_N(t; \Pi_N x, u) - y(t; x,u)\|_{\cH} = 0, \quad \forall \; x \in \cH,
\ee
where $\mathcal{U}_{ad}$ denotes a set of admissible controls; see \cite[Theorem 4.2]{Ferretti97}.

When the evolution equation involves  state- or control-dependent nonlinear terms, the conditions provided in \cite[Proposition 2.1]{Ferretti97} to ensure \eqref{uniform_in_u_conv_Intro} needs to be amended. Whether the nonlinear terms involve the controls or the system's state,  our working assumption regarding the linear terms of the original equation and of its Galerkin approximations, is (as in \cite{Ferretti97}) to satisfy respectively the  ``stability'' and ``consistency'' conditions required in the Trotter-Kato theorem; see Assumption {\bf (A1)} and Assumption {\bf (A2)} in Sect.~\ref{Sec_Galerkin_prelim}.  In the case of a linear equation with nonlinear control terms, a simple compactness assumption about the set of admissible controls (see Assumption {\bf (A5)} in Sect.~\ref{Sect_uniform_convergence}) is sufficient to 
ensure \eqref{uniform_in_u_conv_Intro}; see Remark \ref{Rmk_F=0}.

In the case of an evolution equation depending nonlinearly on the system's state and on the controls, a key assumption is introduced to ensure \eqref{uniform_in_u_conv_Intro} that adds up to standard local Lipschitz conditions on the state-dependent nonlinear terms (Assumption~{\bf (A3)}) and the nonlinear control operator $\mathfrak{C}:V\rightarrow \cH$, where $V$ denotes an auxiliary Hilbert space in which the controls take values. Introducing $\Pi^\perp_N:=\mbox{Id}_{\cH} - \Pi_N$, this assumption concerns a double 
uniform convergence about the residual energy $\|\Pi^\perp_N y(t;x,w)\|_{\cH}$ (see Assumption {\bf (A7)} in Sect.~\ref{Sect_uniform_convergence}), namely
\be \label{Est_uniformity_highmodes_intro}
\lim_{N \rightarrow \infty}  \sup_{u\in \mathcal{U}_{ad}} \sup_{t\in[0,T]} \|\Pi^\perp_N y(t;x,u)\|_{\cH} =0.
\ee

 With this assumption at hand, and the rest of our working assumptions, standard {\it a priori} bounds --- uniform in $u$ in $\mathcal{U}_{ad}$ (Assumption~{\bf (A6)}) --- allow us to ensure
\eqref{uniform_in_u_conv_Intro} for a broad class of nonlinear evolution equations in Hilbert spaces.  The pointwise convergence of the value functions associated with the optimal control problem of the corresponding Galerkin approximations is then easily derived for a broad class of cost functionals; see Theorem~\ref{Thm_cve_Galerkin_val}.

The relevance of assumption \eqref{Est_uniformity_highmodes_intro} for applications is addressed through various angles. First, from the proof of Corollary \ref{Lem:uniform_conv_locally_Lip_t_v2} (and thus Theorem \ref{Lem:uniform_in_u_conv}) in which Theorem~\ref{Thm_cve_Galerkin_val} relies.  In that respect, a sort of pedagogical detour is made in Sect.~\ref{Sect_local_convergence} in which we show essentially that a weaker (than \eqref{uniform_in_u_conv_Intro}) local-in-$u$ approximation  result (Lemma \ref{Lem_local_in_u_est})  follows from the rest of our working assumptions (except Assumptions {\bf (A6)} and {\bf (A7)}\footnote{More precisely, by assuming a weaker version of Assumption {\bf (A6)}, namely Assumption~{\bf (A4)}, and without assuming {\bf (A7)}.}) and from a local-in-$u$ estimate about the residual energy (Lemma~\ref{Lem:local_in_u_conv}); the latter resulting itself from the continuity of the mapping $u\mapsto y(t;x,u)$. Condition \eqref{Est_uniformity_highmodes_intro} constitutes thus a natural strengthening of inherent properties to the approximation problem. 

From a more applied perspective, sufficient conditions concerning the spectrum of the linear part---such as self-adjointness and compact resolvent---are pointed out in Sect.~\ref{Sect_examples} to ensure \eqref{Est_uniformity_highmodes_intro}; see\footnote{See also \cite[Sect.~2.3]{CKL17_DDE} for other spectral conditions which do not rely on self-adjointness while ensuring  \eqref{Est_uniformity_highmodes_intro}.} Lemma \ref{Lem_examples} and  Remark \ref{Rmk_other_spectral_assumption}.  
Finally,  Sect.~\ref{Sec_Err_estimates} provides error estimates concerning the value function and the optimal control that complete the picture and emphasize from another perspective the relevance of the residual energy in the analysis of the approximation problem; see Theorem \ref{Thm_PM_val} and Corollary \ref{Lem_controller_est}.

With this preamble in mind, we provide now the more formal organization of this article. 
First, we present in Sect.~\ref{Sec_Galerkin_prelim} the type of state equation and its corresponding Galerkin approximations that we will be working on. A trajectory-wise convergence result for each fixed control $u$ is then derived in Sect.~\ref{Sect_trajectory_convergence}. As mentioned earlier, it relies  essentially on the theory of $C_0$-semigroups and the Trotter-Kato theorem \cite[Thm.~4.5, p.88]{Pazy83}; see Lemma~\ref{Lem:local_in_u_conv}. In a second step, we derive a ``local-in-$u$'' approximation result in Sect.~\ref{Sect_local_convergence} for controls that lie within a neighborhood of a given control $u$; see Lemma \ref{Lem_local_in_u_est}.  As discussed above, a key approximation property about the residual energy of solutions (see \eqref{Est_local_uniformity_highmodes}) is then amended into  an assumption (see Assumption {\bf (A7)}) 
to ensure a uniform-in-$u$ convergence result; see Theorem \ref{Lem:uniform_in_u_conv} of Sect.~\ref{Sect_uniform_convergence}. As shown in Sect.~\ref{Sec_cve}, this uniform convergence result helps us derive---in the spirit of dynamic programming (see Corollary \ref{Lem:uniform_conv_locally_Lip_t_v2})---the convergence of the value functions associated with optimal control problems based on Galerkin approximations; see Theorem~\ref{Thm_cve_Galerkin_val}. For this purpose, some standard sufficient conditions for the existence of optimal controls are also recalled in Appendix~\ref{Sec_existence_opt_contr}. Simple and useful error estimates about the value function and the optimal control are then provided in Sect.~\ref{Sec_Err_estimates}. In Sect.~\ref{Sect_examples} we point out a broad class of evolution equations for which Assumption {\bf (A7)} is satisfied. 

As applications of the theoretical results derived in Sect.~\ref{Sect_Galerkin}, we show in Sect.~\ref{Sec_appl} that our framework allows to provide rigorous Galerkin approximations to the optimal control of  a broad class of semilinear heat problems, posed on a compact (smooth) manifold without boundary.  As a concrete example, the framework is shown to apply to geoengineering and the mitigation of greenhouse gas (GHG) emissions formulated for the first time here in terms of optimal control of energy balance models (EBMs) arising in climate modeling; see \cite{Budyko69,Sellers69,Ghil76,North_al81} for an introduction on EBMs, and \cite{Diaz98,bermejo2009mathematical} for a mathematical analysis.   After recalling some fundamentals of differential geometry in Sect.~\ref{Sect_diff_geo_prelim} to prepare the analysis, a general convergence result of Galerkin approximations to controlled semilinear heat problems posed on an $n$-dimensional sphere $\mathbb{S}^n$ is formulated in Sect.~\ref{Sec_optctr_man}; see Corollary~\ref{Cor:heat_Galerkin_approx}. The application to the optimal control of EBMs is then presented in Sect.~\ref{Sec_EBM} in the context of geoengineering and GHG emissions for which approximation of the value function and error estimates about the optimal control are obtained. Finally, Sect.~\ref{Sect_concluding_remarks} outlines several possible directions for future research the framework introduced in this article opens up.

\section{Galerkin approximations for optimal control problems: Convergence results}  \label{Sect_Galerkin}
%%%%%%%%%%%%%%%%%%%%%%%%%%
We present in this section, rigorous convergence results for semi-discretization of optimal control problems based on Galerkin approximations. 
In particular, we derive the pointwise convergence of the value functions associated with optimal control problems based on Galerkin approximations  in Sect.~\ref{Sec_cve}.

\subsection{Preliminaries} \label{Sec_Galerkin_prelim}
We consider in this article finite-dimensional approximations of the following initial-value problem (IVP):
\bea \label{ODE}
\frac{\d y}{\d t} &= L y + F(y) + \mathfrak{C} (u(t)),  \quad t \in (0, T],\\
y(0) &= x,
\eea
where $x$ lies in $\cH$, and $\cH$ denotes a separable Hilbert space. The time-dependent forcing $u$ lives in a separable Hilbert space $V$ (possibly different from $\cH$); the (possibly nonlinear) mapping $\mathfrak{C}: V \rightarrow \cH$ is assumed to be such that $\mathfrak{C}(0)=0$.  Other assumptions regarding $\mathfrak{C}$ will be made precise when needed.

We assume that the linear operator $L: D(L) \subset \cH \rightarrow \cH$ is the  infinitesimal  generator of a
$C_0$-semigroup of bounded linear operators $T(t)$ on $\cH$. Recall that in this case the domain $D(L)$ of $L$ is dense in $\cH$ and that $L$ is a closed operator;  see \cite[Cor.~2.5, p.~5]{Pazy83}. 

Under the above assumptions on the operator $L$, recall that there exists $M\geq 1$ and $\omega \geq 0$  \cite[Thm.~2.2, p.~4]{Pazy83} such that
\be\label{Eq_control_T_t}
\|T(t)\| \le M e^{\omega t},  \qquad t \ge 0,
\ee
where $\|\cdot \|$ denotes the operator norm subordinated to $\|\cdot\|_{\cH}$.

For the moment, we take the set of admissible controls to be
\be \label{Eq_admissible_set_LpV}
 \mathcal{U} := L^q(0, T; V),
\ee
with $q\geq1$. In the later subsections, further assumptions on the admissible controls will be specified when needed.

Let $u$ be in $\mathcal{U}$  given by \eqref{Eq_admissible_set_LpV}, a {\it mild solution} to \eqref{ODE} over $[0,T]$ is a function $y$ in $C([0,T],\cH)$ such that
\be\label{Eq_mild}
y(t)=T(t)x + \int_0^t T(t-s) F(y(s)) \d s + \int_0^t T(t-s) \mathfrak{C}(u(s)) \d s, \;\; t\in [0,T].
\ee
In what follows we will often denote by $t\mapsto y(t;x,u)$ a mild solution to \eqref{ODE}.

Let $\{\cH_N: N \in \mathbb{Z}_+^\ast\}$ be a sequence of finite-dimensional subspaces of $\cH$ associated with {\it orthogonal  projectors}
\be
\Pi_N: \cH \rightarrow \cH_N,
\ee
such that 
\be\label{Eq_identity_approx}
\|(\Pi_N-\mbox{Id}) x\| \underset{N \rightarrow \infty}\longrightarrow 0, \qquad \Forall x \in \cH,
\ee
and
\be\label{Eq_XN_in_domain}
\cH_N\subset D(L), \; \forall \, N\geq 1.
\ee

The corresponding Galerkin approximation of \eqref{ODE} associated with $\cH_N$ is then given by:
\bea \label{ODE_Galerkin}
\frac{\d y_N}{\d t} &= L_N y_N + \Pi_N F(y_N) + \Pi_N \mathfrak{C} (u(t)), \; t\in [0,T],\\
y_N(0) &= \Pi_N x, \; \; x\in \cH,
\eea
where  
\be\label{Def_LN}
L_N := \Pi_N L \Pi_N : \cH \rightarrow \cH_N. 
\ee
In particular, the domain 
$D(L_N)$ of $L_N$ is $\cH$, because of \eqref{Eq_XN_in_domain}.

Throughout this section,  we assume the following set of assumptions:
\bi

\item[{\bf (A0)}]  {\it The linear operator $L: D(L) \subset \cH \rightarrow \cH$ is the  infinitesimal  generator of a
$C_0$-semigroup of bounded linear operators $T(t)$ on $\cH$.}

\item[{\bf (A1)}]  {\it For each positive integer $N$, the linear flow $e^{L_N t}:\cH_N \rightarrow \cH_N$ extends to a $C_0$-semigroup $T_N(t)$ on $\cH$. Furthermore the following uniform bound is satisfied by the family $\{T_N(t)\}_{N \geq 1, t\geq0}$
\be \label{Eq_control_linearflow}
\quad \|T_N(t)\| \le M e^{\omega t}, \quad N \geq 1, \; \quad t \ge 0,
\ee
where $\|T_N(t)\|:=\sup\{\|T_N(t)x\|_{\cH}, \; \|x\|_{\cH} \leq 1, x\in \cH\}$ and the constants $M$ and $\omega$ are the same as given in \eqref{Eq_control_T_t}.}

\item[{\bf (A2)}] 
{\it The following convergence holds}
\be \label{Eq_L_Approx}
\lim_{N \rightarrow \infty} \|L_N  \phi - L \phi \|_{\cH} = 0, \quad \Forall \phi \in D(L).
\ee

\item[{\bf (A3)}] {\it The nonlinearity $F$ is locally Lipschitz in the sense given in \eqref{Local_Lip_cond} below.} 

\ei

Following the presentation commonly  adopted for the Trotter-Kato approach, the assumptions {\bf (A0)}-{\bf (A2)} are concerned with the linear parts of the original system \eqref{ODE} and of its Galerkin approximation \eqref{ODE_Galerkin}.
Assumption {\bf (A3)} is concerned with the nonlinearity in \eqref{ODE}. Other assumptions regarding the latter will be made in the sequel.  Throughout this article, a mapping $f: \mathcal{W}_1 \rightarrow \mathcal{W}_2$ between two Banach spaces, $\mathcal{W}_1$ and $\mathcal{W}_2$, is said to be locally Lipschitz if for any ball $\mathfrak{B}_r \subset \mathcal{W}_1$ with radius $r>0$ centered at the origin,  there exists a constant $\Lip(f\vert{_{\mathfrak{B}_r}})>0$ such that 
\be  \label{Local_Lip_cond}
\|f(y_1) - f(y_2)\|_{\mathcal{W}_2} \le \Lip(f\vert{_{\mathfrak{B}_r}})\|y_1 - y_2\|_{\mathcal{W}_1}, \qquad \Forall y_1, y_2 \in \mathfrak{B}_r.
\ee

\subsection{Convergence of Galerkin approximations: Trajectory-wise result} \label{Sect_trajectory_convergence}   
As a preparation for the main result given in Sect.~\ref{Sec_cve}, we derive hereafter a trajectory-wise convergence result for the solutions to the Galerkin approximations \eqref{ODE_Galerkin}; see Lemma \ref{Lem:local_in_u_conv} below.

With this purpose in mind, besides {\bf (A0)}--{\bf (A3)}, we will also make use of the following assumption. 
\bi
\item[{\bf (A4)}] {\it For each $T>0$, $x$ in $\cH$ and each $u$ in $\mathcal{U}$ with $\mathcal{U}$ defined in \eqref{Eq_admissible_set_LpV}, the problem \eqref{ODE} admits a unique mild solution $y(\cdot; x, u)$ in $C([0,T],\cH)$, and its Galerkin approximation \eqref{ODE_Galerkin} admits a unique solution $y_N(\cdot; \Pi_N x, u)$ in  $C([0,T],\cH_N)$ for each $N$ in $\mathbb{Z}_+^\ast$. Moreover, there exists a constant $C:=C(T,x,u)$ such that}
\bea \label{Eq_y_local-in-u_bounds}
& \|y_N(t; \Pi_N x, u)\|_{\cH} \le C, \qquad \Forall t\in[0,T],\; N \in \mathbb{Z}_+^\ast.
\eea
\ei

Note that in applications, {\bf (A4)} is typically satisfied via {\it a priori} estimates; see Remark~\ref{Rmk:unif_bdd_on_soln}-(ii) below. 

\bl \label{Lem:local_in_u_conv}
Let $\mathcal{U}$ be the set of admissible controls given by \eqref{Eq_admissible_set_LpV}, with $q$ defined therein. 
Consider the IVP \eqref{ODE} and the associated Galerkin approximation \eqref{ODE_Galerkin}.  Assume that the nonlinear operator $\mathfrak{C}:V\rightarrow \cH$ satisfies, for  $1\leq p \leq q$, the following growth condition, 
\be \label{Eq_poly_growth_C}
\|\mathfrak{C} (w)\|_{\cH} \le \gamma_1 \|w\|^p_V + \gamma_2, \qquad \Forall w \in V, 
\ee
where $\gamma_1 > 0$ and $\gamma_2 \ge 0$.  

Assume also that {\bf (A0)}--{\bf (A4)} hold. Then for each fixed $T > 0$, $x$  in ${\cH}$ and $u$ in $\mathcal{U}$, the following convergence result is satisfied:
\be \label{local_in_u_conv_Goal}
\lim_{N\rightarrow \infty}  \sup_{t \in [0, T]} \|y_N(t; \Pi_N x, u) - y(t; x,u)\|_{\cH} = 0.
\ee

\el

\bp
To simplify the notations, only the $x$-dependency that matters to the estimates derived hereafter will be made explicit.  
Let $u$ be given in $\mathcal{U}$ and $x$ in $\cH$, then by the variation-of-constants formula applied to Eq.~\eqref{ODE_Galerkin} we have, for $0\leq t\leq T$ and $N$ in $\mathbb{Z}_+^\ast$, that
\bea \label{Eq_Galerkin_VCF}
y_N(t; u) = e^{L_N t} \Pi_N x & + \int_0^t e^{L_N (t -s )} \Pi_N F(y_N(s; u)) \d s \\
& \hspace{4em} + \int_0^t e^{L_N (t -s )} \Pi_N \mathfrak{C} (u(s)) \d s.
\eea

Then, it follows from \eqref{Eq_mild} and \eqref{Eq_Galerkin_VCF} that the difference 
\bes
w_N(t; u):=y(t; u)  - y_N(t; u), 
\ees
satisfies 
\bea
w_N(t; u)&=  T(t) x - e^{L_N t} \Pi_N x  \\
&\hspace{2em}+ \int_0^t T(t-s)F(y(s; u)) \d s - 
\int_0^t e^{L_N (t -s )} \Pi_N F(y_N(s; u)) \d s  \\
&  \hspace{4em}  +  \int_0^t T(t-s)\mathfrak{C} (u(s)) \d s - \int_0^t e^{L_N (t -s )} \Pi_N \mathfrak{C} (u(s)) \d s \\
\eea
and hence, we have 
\bea
w_N(t; u)&=  T(t) x - e^{L_N t} \Pi_N x +  
\int_0^t \big( T(t-s) -  e^{L_N (t -s )} \Pi_N  \big)  F(y(s; u)) \d s  \\
& \hspace{3em} +  \int_0^t e^{L_N (t -s )} \Pi_N \big(F(y(s; u)) - F(y_N(s; u)) \big ) \d s \\
& \hspace{6em}+ \int_0^t \Big( T(t-s) - e^{L_N (t -s )} \Pi_N \Big) \mathfrak{C} (u(s)) \d s. \label{eq:residual}
\eea

Let us introduce for every $s$ in $[0, T]$, $x$ in $\cH$, and $u$ in $\mathcal{U}$,
\begin{subequations}\label{eq:symbols_residual}
\begin{align}
 r_N(s; u) & := \|y(s; u)  - y_N(s; u) \|_{\cH}, \;\;  \label{totoa} \\
\zeta_N(x) & := \sup_{t\in[0, T]} \|T(t) x - e^{L_N t} \Pi_N x\|_{\cH}, \;\;  \label{totob}\\ 
 d_N(s; u) & := \sup_{t\in[s, T]} \| \big( T(t-s) -  e^{L_N (t -s )} \Pi_N \big) F(y(s; u)) \|_{\cH},\label{totoc} 
\end{align} 
 \end{subequations}
and for almost every $s$ in $[0, T]$,
\be\label{totod}
\widetilde{d}_N (s; u) := \sup_{t\in[s, T]} \| \big( T(t-s) -  e^{L_N (t -s )} \Pi_N \big) \mathfrak{C}(u(s)) \|_{\cH}. 
\ee

For $x$ in $\cH$ and $u$ in $\mathcal{U}$, let us denote by $\mathfrak{B}$ the closed ball in $\cH$ with radius $C$ centered at the origin where $C$ is the upper bound in estimate \eqref{Eq_y_local-in-u_bounds} for the Galerkin solutions (Assumption {\bf (A4)}). 

Since by assumption, $t\mapsto y(t; x, u)$ lies in $ C([0,T],\cH)$ for any $x$ in $\cH$ and $u$ in $\mathcal{U}$ given by \eqref{Eq_admissible_set_LpV}, one can assume without loss of generality that $y(t; x, u)$ stays in $\mathfrak{B}$ for all $t$ in $[0,T]$, by possibly redefining $C$.

We obtain  then  from \eqref{eq:residual} that
\bea \label{eq:residual2}
 r_N(t; u)  & \le \zeta_N(x) + \int_0^t \|e^{L_N (t -s )} \Pi_N \big(F(y(s; u)) -F(y_N(s; u)) \big )\|_{\cH} \d s \\ 
 & \hspace{12.5em} +  \int_0^t d_N(s; u) \d s +\int_0^t  \widetilde{d}_N(s; u) \d s\\
& \le \zeta_N(x) + \int_0^t d_N(s; u) \d s + M \mbox{Lip}(F\vert{_\mathfrak{B}}) \int_0^t e^{\omega (t -s )}  r_N(s; u) \d s \\
 & \hspace{20em}+   \int_0^t  \widetilde{d}_N(s; u) \d s\\
& \le \zeta_N(x) + \int_0^t d_N(s; u) \d s + M\mbox{Lip}(F\vert{_\mathfrak{B}})  \, e^{\omega T} \int_0^t  r_N(s; u) \d s  \\
 & \hspace{20em} +   \int_0^t  \widetilde{d}_N(s; u) \d s,
\eea
where we have used that $F$  is locally Lipschitz (Assumption {\bf (A3)}) and the uniform bound  \eqref{Eq_control_linearflow}  of  Assumption~{\bf (A1)}.

It follows then from Gronwall's inequality that for all $t$ in $[0,T]$,
\be \label{eq_rn_gronwall}
\hspace{-.95ex}r_N(t; u) \le \Big(\zeta_N(x)  +  \hspace{-.33ex}\int_0^T  \hspace{-.33ex} d_N(s; u) \d s +  \int_0^T  \hspace{-.33ex} \widetilde{d}_N(s; u) \d s\Big) \exp\big(M \mbox{Lip}(F\vert{_\mathfrak{B}}) e^{\omega T} T \big).
\ee

We are thus left with the estimation of $\zeta_N(x)$,  $\int_0^T d_N(s; u) \d s$ and $\int_0^T  \widetilde{d}_N(s; u) \d s$ as $N \rightarrow \infty$. Note that the assumptions {\bf (A1)} and {\bf (A2)} allow us to use a version of the Trotter-Kato theorem \cite[Thm.~4.5, p.88]{Pazy83} which implies 
together with \eqref{Eq_identity_approx} that 
\be \label{eq:linear_pointwise_convergence}
\lim_{N\rightarrow \infty} e^{L_N t} \Pi_N \phi  = T(t) \phi, \quad \Forall \phi \in \cH,
\ee
uniformly in $t$ lying in bounded intervals; see also \cite[footnote 6]{CGLW15}.

It follows that 
\be  \label{eq:est_initial}
\lim_{N\rightarrow \infty} \zeta_N(x)=0,
\ee
and that $d_N(\cdot; u)$ and $\widetilde{d}_N(\cdot; u)$ converge point-wisely to zero on $[0,T]$, i.e.,
\be \label{eq:dn_pointwise_convergence}
\lim_{N\rightarrow \infty} d_N(s; u)  = 0, \qquad \Forall s \in [0,T],
\ee
and 
\be \label{eq:tilde_dn_pointwise_convergence}
\lim_{N\rightarrow \infty} \widetilde{d}_N(s; u)  = 0,  \textrm{ for a.e. } s \in [0,T].
\ee

On the other hand, by using \eqref{Eq_control_T_t} and \eqref{Eq_control_linearflow} and from the local Lipschitz assumption on $F$, we get for $0\le s \le t \le T$,
\bea
\| \big( T(t-s) -  e^{L_N (t -s )} & \Pi_N \big) F(y(s; u)) \|_{\cH} \\
& \le 2Me^{\omega (t-s)}\|F(y(s; u)) \|_{\cH} \\
& \le 2Me^{\omega (t-s)}\Big(\|F(y(s; u)) - F(0) \|_{\cH} +\|F(0) \|_{\cH} \Big) \\
& \le 2Me^{\omega (t-s)} \big(\mbox{Lip}(F\vert{_\mathfrak{B}}) \, \|y(s; u) \|_{\cH}  + \|F(0)\|_{\cH}\big), 
\eea
which implies 
\be \label{eq:dn_dominated}
d_N(s,u) \le 2Me^{\omega T} \big(\mbox{Lip}(F\vert{_\mathfrak{B}})  \, \|y(s; u) \|_{\cH}  + \|F(0)\|_{\cH}\big), \qquad \Forall s \in [0,T]. 
\ee
Since $y(\cdot; u)$  lies in $C([0,T]; \cH)$, the mapping $s\mapsto \|y(s,u)\|_{\cH}$ is in particular integrable on $[0,T]$, and the Lebesgue dominated convergence theorem allows us to conclude from \eqref{eq:dn_pointwise_convergence} and \eqref{eq:dn_dominated} that 
\be \label{eq:est_dn_vf}
\lim_{N\rightarrow \infty} \int_0^T d_N(s; u) \d s  =0.
\ee

Let us estimate  $\int_0^T \widetilde{d}_N(s; u) \d s$ as $N \rightarrow \infty$.  By using the growth condition \eqref{Eq_poly_growth_C}, we get for a.e. $s \in [0, T]$ and all $t \in [s, T]$
\be \label{eq:tilde_dn_dominated}
\| \big( T(t-s) -  e^{L_N (t -s )} \Pi_N \big) \mathfrak{C}(u(s)) \|_{\cH} \le 2 M e^{\omega T} \Big(\gamma_1 \|u(s)\|^p_V + \gamma_2\Big).
\ee
Since $u$ lies in $L^q([0,T];V)$ with $q\geq p$, then $u$ lies in $L^p([0,T];V)$ and  the right hand side (RHS) of \eqref{eq:tilde_dn_dominated} is integrable on $[0,T]$. The Lebesgue dominated  convergence theorem allows us then to conclude from \eqref{eq:tilde_dn_pointwise_convergence} and \eqref{eq:tilde_dn_dominated} that
\be \label{eq:est_tilde_dn_vf}
\lim_{N\rightarrow \infty}  \int_0^T \widetilde{d}_N(s; u)  \d s  =0.
\ee
The desired convergence result \eqref{local_in_u_conv_Goal} follows now from \eqref{eq_rn_gronwall} by using \eqref{eq:est_initial}, \eqref{eq:est_dn_vf} and \eqref{eq:est_tilde_dn_vf}.
\ep

%%%%%%%%%%%%%%%%%%%%%%%%%%%%%%%%%%%%%%%%%%%%%%%%%%%%%%%%%%%%
\subsection{Local-in-$u$ approximation result} \label{Sect_local_convergence}  

In this section, we consider for $q\geq 1$, $\mathcal{U}_{ad}$, to be the subset of $L^q(0,T; V)$, constituted by measurable functions that 
take values in  $U$, a bounded subset of the Hilbert space $V$. 
In other words, 
\be\label{Eq_U_bounded}
\mathcal{U}_{ad}:=\{f\in L^q(0,T; V)\;:  \; f(s) \in U \textrm{ for a.e. } s \in [0,T]\},\; \;q\geq 1.
\ee
The set $\mathcal{U}_{ad}$ will be endowed with the induced topology from that of $L^q(0,T; V)$.

We present hereafter a natural property that is derived from our working assumptions, namely that given an finite-dimensional approximation $\cH_N$ of $\cH$, the residual energy of solutions to the IVP \eqref{ODE}, i.e.~
\bes
\|(\mbox{Id}_{\cH} - \Pi_N) y(t;x,w)\|_{\cH},
\ees
 can be made arbitrarily small as $N\rightarrow \infty$ and uniformly in $w$, provided that $w$ lies within a sufficiently small open set of $\mathcal{U}_{ad}$ given in \eqref{Eq_U_bounded}.

This is the purpose of Lemma \ref{Lem_A7_local_verion} which boils down to proving the continuity of the mapping $u\mapsto y(t;x,u)$; see \eqref{local_uniformity_highmodes_est1} below. As a consequence a local-in-$u$ approximation result is naturally inferred; see Lemma \ref{Lem_local_in_u_est}. However, as pointed out and amended in Sect.~\ref{Sect_uniform_convergence} below, this is insufficient to guarantee convergence results for the value functions associated with Galerkin approximations of optimal control problems subordinated to \eqref{ODE}.

The merit of Lemma \ref{Lem_A7_local_verion} below is nevertheless not only to identify a symptom, but also to help us propose a cure.
Indeed, by requiring the residual energy of the solution to the IVP \eqref{ODE} to vanish uniformly (in $u$) as $N\rightarrow \infty$, we are able to conclude about the desired convergence results for the value functions. The latter uniform property (i.e.~the ``cure'') is shown below to hold for a broad class of IVPs; see Sect.~\ref{Sect_examples}.  For the moment, let us present the ``symptom,'' i.e.~the local-in-$u$ approximation results.  For that purpose, we start with a local-in-$u$
 estimate about the residual energy,  
 \bes
 \|(\mbox{Id}_{\cH} - \Pi_N) y(t;x,w)\|_{\cH},
\ees
where $w$ lives in some neighborhood of $u.$

\bl \label{Lem_A7_local_verion}

Assume that {\bf (A0)}--{\bf (A4)} hold. Assume furthermore that $\mathfrak{C}:V\rightarrow \cH$ is locally Lipschitz  and that the admissible controls lie in $\mathcal{U}_{ad}$ given by \eqref{Eq_U_bounded} and endowed with the induced $ L^q(0,T; V)$-topology, for $q > 1$.

Then, for any $(x,u)$ in $\cH \times \mathcal{U}_{ad}$  and any $\epsilon>0$ there exists a neighborhood $\mathcal{O}_u \subset \mathcal{U}_{ad}$ of $u$ and $N_0\geq 1$ such that the mild solution $y(t;x,u)$ to \eqref{ODE}  satisfies
\be \label{Est_local_uniformity_highmodes}
\sup_{w \in \mathcal{O}_u} \sup_{t\in[0,T]} \|(\mbox{Id}_{\cH} - \Pi_N) y(t;x,w)\|_{\cH} \le \epsilon, \qquad \Forall N \ge N_0.
\ee

\el

\bp

As we will see, we mainly need to show that the solution $y(t;x,u)$ to the IVP \eqref{ODE} depends continuously on the control $u$ in $\mathcal{U}_{ad}$, endowed with the $L^q(0,T; V)$-topology. The above estimate \eqref{Est_local_uniformity_highmodes} follows then directly from this continuous dependence as explained at the end of the proof.

Given $u$  in $\mathcal{U}_{ad}$ and $r > 0$, we denote by $\mathfrak{B}_{\scriptscriptstyle\mathcal{U}_{ad}}(u,r)$ the closed ball  centered at $u$ with radius $r$, for the induced $ L^q(0,T; V)$-topology on $\mathcal{U}_{ad}$.   For any $w$ in $\mathfrak{B}_{\scriptscriptstyle\mathcal{U}_{ad}}(u,r)$, $x$ in $\cH$ and $t$ in $[0,T]$, consider $\phi(t):=y(t;x, u)  - y(t;x,w)$ and note that $\phi(0)=0$. It follows from \eqref{Eq_mild} that
\bea 
\phi(t) &=  \int_0^t T(t-s)\big(F(y(s;x,u)) - F(y(s;x,w)) \big)\d s \\
& \hspace{6em}+  \int_0^t T(t-s) \big( \mathfrak{C} (u(s)) - \mathfrak{C} (w(s))\big) \d s. 
\eea
We have then
\bea \label{local_uniformity_highmodes_est0}
\|\phi(t)\|_{\cH} & \leq \int_0^t M e^{\omega(t-s)} \|F(y(s;x,u)) - F(y(s;x,w))\|_{\cH} \d s \\
& \hspace{6em} + \int_0^t  M e^{\omega(t-s)} \|\mathfrak{C} (u(s)) - \mathfrak{C} (w(s))\|_{\cH} \d s.
\eea
Let $C>0$ be chosen such that $\|y(t;x,u)\|_{\cH} \le C$ for all $t\in [0,T]$, and let  $w$ be in $\mathfrak{B}_{\mathcal{U}_{ad}}(u,r)$, we define then
\be \label{Def_t*}
t^*:=\max\{t \in [0,T] \; : \; \|y(t;x,w)\|_{\cH} < 2 C\}.
\ee

First let us note that $t^\ast >0$. Indeed recalling that $ \|y(t;x,u)\|_{\cH} \le C$ by assumption, we have in particular that 
$\|x\|_{\cH} \le C$. Now due to the continuity for any $(x,w)$ in $\cH \times \mathcal{U}_{ad}$ of the mapping 
\beas
&[0,T] \longrightarrow \cH\\
&t\longmapsto y(t;x,w),
\eeas
(since $y(t;x,w)$ is a mild solution), we infer, since $\|x\|_{\cH} \le C$, for each $w$ in $\mathfrak{B}_{\scriptscriptstyle\mathcal{U}_{ad}}(u,r)$ the existence of $t'(w)>0$ such that 
\bes
\|y(t;x,w)\|_{\cH} < 2 C,  \; \Forall t \in [0,t'(w)],
\ees
and therefore $t^\ast >0$.

Denote also $\mathfrak{B}_{ \scriptscriptstyle\cH} \subset \cH$ the closed ball centered at the origin with radius $2C$. 
Let $\mathfrak{B}_{\scriptscriptstyle V}$ be the smallest closed ball in $V$ containing the bounded set $U$.
By using the local Lipschitz property of $F$ and $\mathfrak{C}$, we obtain from \eqref{local_uniformity_highmodes_est0} that
\bea
\hspace{-1ex}\|\phi(t)\|_{\cH} & \le \int_0^t M e^{\omega(t-s)} \Lip(F\vert_{\mathfrak{B}_{\scriptscriptstyle \cH}})\|\phi(s)\|_{\cH} \d s \\
& \hspace{5em} + \int_0^t  M e^{\omega(t-s)} \Lip(\mathfrak{C}\vert_{\mathfrak{B}_{\scriptscriptstyle V}}) \|u(s) - w(s)\|_{V} \d s \\
& \le M\Lip(F\vert_{\mathfrak{B}_\cH}) \int_0^t e^{\omega(t-s)} \|\phi(s)\|_{\cH} \d s \\
& \hspace{5em} + e^{\omega t^*} M \Lip(\mathfrak{C}\vert_{\mathfrak{B}_V}) \int_0^t  \|u(s) - w(s)\|_{V} \d s,  \; \Forall t \in [0, t^*].
\eea
By H\"older's inequality, we have 
\be
\int_0^t  \|u(s) - w(s)\|_{V} \d s \le t^{\frac{p-1}{p}}\|u-w\|_{L^p(0,T;V)} \le t^{\frac{p-1}{p}} r, 
\ee
which leads to
\bea
\|\phi(t)\|_{\cH}& \le M\Lip(F\vert_{\mathfrak{B}_{\scriptscriptstyle \cH}}) \int_0^t e^{\omega(t-s)} \|\phi(s)\|_{\cH} \d s\\
& \hspace{6em}  + (t^*)^{\frac{p-1}{p}} r e^{\omega t^*} M\Lip(\mathfrak{C}\vert_{\mathfrak{B}_{\scriptscriptstyle V}}),  \; \Forall t \in [0, t^*].
\eea

It follows then from Gronwall's inequality that
\bea  \label{local_uniformity_highmodes_est0_b}
\hspace{-1em}\|\phi(t)\|_{\cH} & \le  (t^*)^{\frac{p-1}{p}} r M \Lip(\mathfrak{C}\vert_{\mathfrak{B}_V}) \exp{(2 \omega t^\ast + t^\ast M\Lip(F\vert_{\mathfrak{B}_\cH}))}\\
& \le  T^{\frac{p-1}{p}} r M \Lip(\mathfrak{C}\vert_{\mathfrak{B}_V})\exp{(2 \omega T + T M\Lip(F\vert_{\mathfrak{B}_\cH}))}\; , \quad \Forall t \in  [0,t^*].
\eea
Now, let $C_\ast:=T^{\frac{p-1}{p}}  M \Lip(\mathfrak{C}\vert_{\mathfrak{B}_V})\exp{(2 \omega T + T M\Lip(F\vert_{\mathfrak{B}_\cH}))}$ and 
\bes  \label{local_uniformity_highmodes_est0_c}
r_1 := \frac{C}{2 C_\ast}.
\ees
We claim that $t^* = T$ if $r \le r_1$. Otherwise, if $t^*< T$, applying \eqref{local_uniformity_highmodes_est0_b} at $t=t^*$, we get
\be
\|\phi(t^*)\|_{\cH} \le \frac{C}{2},
\ee
which leads then to 
\be
\|y(t^*;x,w)\|_{\cH}  \le \|y(t^*;x,u)\|_{\cH} + \|\phi(t^*)\|_{\cH} \le \frac{3}{2}C.
\ee
This last inequality contradicts with the definition of $t^*$ given by \eqref{Def_t*}. 

We obtain thus for each $r \in (0, r_1]$ that
\be  \label{local_uniformity_highmodes_est0_c}
\|\phi(t)\|_{\cH} \le r C_\ast\; , \qquad \Forall t \in  [0,T].
\ee
Now, it follows from \eqref{local_uniformity_highmodes_est0_c} that for any fixed $\epsilon > 0$ there exists $r_\epsilon > 0$ sufficiently small such that
\be
\|\phi(t)\|_{\cH} \le \frac{1}{2} \epsilon, \qquad \Forall t\in[0,T].
\ee
Recalling the definition of $\phi$, we have thus proved that
\be \label{local_uniformity_highmodes_est1}
\sup_{w\in \mathfrak{B}_{\scriptscriptstyle\mathcal{U}_{ad}}(u,r_\epsilon)} \sup_{t\in[0,T]}\|y(t;x,u) - y(t;x,w)\|_{\cH} \le \frac{1}{2}\epsilon.
\ee

We turn now to the last arguments needed to prove \eqref{Est_local_uniformity_highmodes}. It consists first to note that the convergence property \eqref{Eq_identity_approx} and the continuity of $t\mapsto y(t; x, u)$ imply, for the given $\epsilon > 0$, the existence of a positive integer  $N_0$ for which 
\be \label{local_uniformity_highmodes_est2}
 \sup_{t\in[0,T]} \|(\mbox{Id}_{\cH} - \Pi_N) y(t;x,u)\|_{\cH} \le \frac{1}{2}\epsilon, \qquad \Forall N \ge N_0.
\ee
Now by defining $\Pi_N^{\perp}:=\mbox{Id}_{\cH} - \Pi_N$, and noting that
\bea
\|\Pi_N^{\perp} y(t;x,w)\|_{\cH}& \leq \|\Pi_N^{\perp} (y(t;x,w)-y(t;x,u))\|_{\cH} +\|\Pi_N^{\perp} y(t;x,u)\|_{\cH}\\
                                               & \leq \|y(t;x,w)-y(t;x,u)\|_{\cH} +\|\Pi_N^{\perp} y(t;x,u)\|_{\cH},
\eea
we conclude---from \eqref{local_uniformity_highmodes_est1} and \eqref{local_uniformity_highmodes_est2}---to the desired estimate \eqref{Est_local_uniformity_highmodes} with $\mathcal{O}_u$ taken  to be $\mathring{\mathfrak{B}}_{\mathcal{U}_{ad}}(u,r_{\epsilon})$, the open ball in $\mathcal{U}_{ad}$ of radius $r_\epsilon$.
\ep

%%%%%%%%%%%%%%%%%%%%%%%%%%%
We conclude this section with a local-in-$u$ approximation result.
\bl \label{Lem_local_in_u_est}
Assume the assumptions of  Lemma~\ref{Lem_A7_local_verion} hold. Then, for any $(x,u)$ in $\cH \times \mathcal{U}_{ad}$  and any $\epsilon>0$ there exists a neighborhood $\mathcal{O}_u \subset \mathcal{U}_{ad}$ of $u$ and $N_0\geq 1$ such that the mild solution $y(t;x,u)$ to \eqref{ODE}  satisfies
\be \label{Eq_local_in_u_est_goal}
\sup_{w \in \mathcal{O}_u} \sup_{t\in[0,T]} \|y_N(t;\Pi_N x,w) - y(t;x,w)\|_{\cH} \le \epsilon, \qquad \Forall N \ge N_0,
\ee
where $y_N$ denotes the solution to the Galerkin approximation \eqref{ODE_Galerkin}.
\el

\bp
First, let us remark that even if here $\mathfrak{C}$ does not satisfy the growth condition \eqref{Eq_poly_growth_C}, one can still derive the trajectory-wise convergence result \eqref{local_in_u_conv_Goal} by exploiting the fact that $u$ lies in $\mathcal{U}_{ad}$ and $\mathfrak{C}$ is locally Lipschitz. 
The only change in the proof consists indeed of replacing (for a.e. $s \in [0, T]$ and all $t \in [s, T]$) the estimate  \eqref{eq:tilde_dn_dominated} by the following: 
\be \label{eq:tilde_dn_dominated_new}
\| \big( T(t-s) -  e^{L_N (t -s )} \Pi_N \big) \mathfrak{C}(u(s)) \|_{\cH} \le 2 M e^{\omega T} \Lip(\mathfrak{C}\vert_{\mathfrak{B}_V}) \|u(s)\|_V, 
\ee
where $\mathfrak{B}_{\scriptscriptstyle V}$ denotes the smallest closed ball in $V$ containing the bounded set $U$. Since $u$ lies in $L^q([0,T];V)$, then $u$ lies in $L^1([0,T];V)$ and  the RHS of \eqref{eq:tilde_dn_dominated_new} is integrable on $[0,T]$. We can then, this time from \eqref{eq:tilde_dn_dominated_new}, still use the Lebesgue dominated  convergence theorem to have \eqref{eq:est_tilde_dn_vf} to hold, without thus assuming the growth condition \eqref{Eq_poly_growth_C}.  

We explain now how to derive \eqref{Eq_local_in_u_est_goal}  from \eqref{local_in_u_conv_Goal}.  
Recall that from \eqref{local_uniformity_highmodes_est1} derived  in the proof of Lemma \ref{Lem_A7_local_verion}, the mild solutions to the IVP \eqref{ODE} depend continuously on the control $u$ in $\mathcal{U}_{ad}$, endowed with the $L^q([0,T];V)$-topology. An estimate similar to \eqref{local_uniformity_highmodes_est1}  can be derived for the solutions to Galerkin approximation \eqref{ODE_Galerkin}, ensuring thus also their continuous dependence on $u$.    

As a consequence, for any given $\epsilon >0$ and $u$ in $\mathcal{U}_{ad}$, there exists a neighborhood $\mathcal{O}_u \subset \mathcal{U}_{ad}$ containing $u$ for which 
\bea  \label{Eq_local_in_u_est1}
& \sup_{w\in \mathcal{O}_u} \sup_{t\in[0,T]}\|y(t;x,u) - y(t;x,w)\|_{\cH} \le \frac{1}{3}\epsilon, \\
& \sup_{w\in \mathcal{O}_u} \sup_{t\in[0,T]}\|y_N(t;\Pi_N x,u) - y_N(t;\Pi_N x,w)\|_{\cH} \le \frac{1}{3}\epsilon, \quad \Forall N \in \mathbb{Z}_+^\ast.
\eea
On the other hand, the trajectory-wise convergence  \eqref{local_in_u_conv_Goal} ensures the existence of a positive integer $N_0$, such that
\be \label{Eq_local_in_u_est2}
\sup_{t \in [0, T]} \|y_N(t; \Pi_N x, u) - y(t; x,u)\|_{\cH} \le \frac{1}{3}\epsilon, \qquad \Forall N \ge N_0.
\ee
Since 
\bea
\|y_N(t;\Pi_N x,w) - y(t;x,w)\|_{\cH} & \le \|y_N(t;\Pi_N x,w) - y_N(t;\Pi_N x,u)\|_{\cH} \\
& \qquad + \|y_N(t;\Pi_N x,u) - y(t; x,u)\|_{\cH} \\
& \qquad  + \|y(t;x,u) - y(t;x,w)\|_{\cH},
\eea
the desired estimate \eqref{Eq_local_in_u_est_goal} follows from \eqref{Eq_local_in_u_est1} and \eqref{Eq_local_in_u_est2}.
\ep

%%%%%%%%%%%%%%%%%%%%%%%%%%%%%%%%%%%%%%%%%%%%%%%%%%%%%%%%%%%%%%%%%%%%%%%%%%%%%%%%%%%%%%%%%%%%%%%%%%%%%%%%%%%%
\subsection{Convergence of Galerkin approximations: Uniform-in-$u$ result} \label{Sect_uniform_convergence}
In the previous section, the local-in-$u$ approximation result has been derived under a boundedness assumption on $U$ arising 
in the definition of $\mathcal{U}_{ad}$. Here, compactness  will substitute the boundedness to derive  convergence results that are uniform in $u$.  More precisely, we will assume for that purpose
\bi
\item[{\bf (A5)}]   {\it The set of admissible controls $\mathcal{U}_{ad}$ is given by \eqref{Eq_U_bounded} with  
$U$ being a compact subset of the Hilbert space~$V$.}
\ei

We will make also use of the following assumptions.
\bi
\item[{\bf (A6)}] {\it Let $\mathcal{U}_{ad}$ be given by  \eqref{Eq_U_bounded}}.  {\it For each $T>0$, $(x,u)$ in $\cH \times \mathcal{U}_{ad}$, the problem \eqref{ODE} admits a unique mild solution $y(\cdot; x, u)$ in $C([0,T],\cH)$, and for each $N\geq 1$, its Galerkin approximation \eqref{ODE_Galerkin} admits a unique solution $y_N(\cdot; \Pi_N x, u)$ in $C([0,T],\cH)$. Moreover, there exists a constant $\mathcal{C}:=\mathcal{C}(T,x)$ such that}
\bea \label{Eq_y_uniform-in-u_bounds}
& \|y(t; x, u)\|_{\cH} \le \mathcal{C}, \qquad \Forall t\in[0,T], \; u \in  \mathcal{U}_{ad}, \\
& \|y_N(t; \Pi_N x, u)\|_{\cH} \le \mathcal{C}, \qquad \Forall t\in[0,T],\; N \in \mathbb{Z}_+^\ast, \; u \in  \mathcal{U}_{ad}.
\eea

\item[{\bf (A7)}] {\it  Let $\mathcal{U}_{ad}$ be given by  \eqref{Eq_U_bounded}. For each fixed $T > 0$, and any mild solution $y(\cdot;x,u)$ to \eqref{ODE} with  $(x,u)$ in $\cH \times \mathcal{U}_{ad}$, it holds that}
\be \label{Est_uniformity_highmodes}
\lim_{N \rightarrow \infty}  \sup_{u\in \mathcal{U}_{ad}} \sup_{t\in[0,T]} \|(\mbox{Id}_{\cH} - \Pi_N) y(t;x,u)\|_{\cH} =0.
\ee
\ei

\br \label{Rmk:unif_bdd_on_soln}
\hspace*{2em}  \vspace*{-0.4em}
\bi
\item[(i)]  Note that {\bf (A6)} differs from {\bf (A4)} by the inequality $\|y(t; x, u)\|_{\cH}\leq \mathcal{C}$, and that 
  the constant in \eqref{Eq_y_uniform-in-u_bounds}  is independent of the control $u$.
\item[(ii)]  Let $u$ be in $\mathcal{U}_{ad}$  given by  \eqref{Eq_U_bounded}. Then uniform bounds such as in \eqref{Eq_y_uniform-in-u_bounds} are guaranteed if e.g.~an {\it a priori} estimate of the following type holds for the IVPs \eqref{ODE} and \eqref{ODE_Galerkin}:
\be
\hspace{5ex}\underset{t \in [0,T]}\sup \| y(t;x,u)\|_{\cH}\leq \alpha ( \| x\|_{\cH}+ \|u\|_{L^{q}(0,T;V)}) +\beta,\;\; \;\; \alpha >0, \;\; \beta\geq 0.
\ee
See e.g.~\cite{Cazenave_al98,Tem97} for such a priori bounds for nonlinear partial differential equations.
Such bounds can also be derived for nonlinear systems of delay differential equations (DDEs); see in that respect the proofs of \cite[Estimates (4.75)]{CGLW15} and \cite[Corollary 4.3]{CGLW15} which can be adapted to the case of controlled DDEs.

\item[(iii)] We refer to Sect.~\ref{Sect_examples} below for a broad class of IVPs for which Assumption {\bf (A7)} holds. 
\ei
\er

As a preparatory lemma to Theorem \ref{Lem:uniform_in_u_conv}, we first prove that given $x$ in $\cH$, 
\bes
(\mbox{Id}_{\cH} - \Pi_N) F(y(t;x,u)) \underset{N\rightarrow \infty}\longrightarrow 0,
\ees
 uniformly in $u$ lying in $\mathcal{U}_{ad}$ and $t$ in $[0,T]$. For that only assumptions {\bf (A3)}, {\bf (A6)} and {\bf (A7)} are used, and $U$ involved in the definition \eqref{Eq_U_bounded} of $\mathcal{U}_{ad}$ is assumed to be bounded (not necessarily compact).  
We have

\bl  \label{Lem:highmode_F_est}
Assume that {\bf (A3)}, {\bf (A6)} and {\bf (A7)} hold. Then, 
\be \label{Eq_highmode_F_est}
\lim_{N\rightarrow \infty} \sup_{u\in \mathcal{U}_{ad}} \sup_{t\in[0,T]} \|(\mathrm{Id}_{\cH} - \Pi_N) F(y(t;x,u))\|_{\cH} =0.
\ee

\el

\bp

For any given $\epsilon > 0$, by {\bf (A7)}, there exists $N_0$ in $\mathbb{Z}_+^\ast$ such that
\be \label{Eq_unif_est_ys}
\sup_{u\in \mathcal{U}_{ad}} \sup_{t\in[0,T]} \|(\mbox{Id}_{\cH} - \Pi_{N}) y(t;x,u)\|_{\cH} \le \epsilon, \quad \Forall N \ge N_0.
\ee
Note also that for all $x$ in $\cH$, and $u$ in $\mathcal{U}_{ad}$ given by  \eqref{Eq_U_bounded}
 \bea \label{Eq_unif_est_F}
\hspace{-2ex} \|(\mbox{Id}_{\cH}  \hspace{-.5ex} -\hspace{-.5ex}  \Pi_N) F(y(t;x,u))\|_{\cH}  &  \hspace{-.5ex}\le  \hspace{-.5ex} \| (\mbox{Id}_{\cH}  \hspace{-.5ex}- \hspace{-.5ex} \Pi_N) F(\Pi_{N_0}  y(t; x, u)) \|_{\cH} \\
& \quad  \hspace{-2ex}+  \| (\mbox{Id}_{\cH}  \hspace{-.5ex}-  \hspace{-.5ex}\Pi_N) \big(F(y(t; x, u)) \hspace{-.5ex} - \hspace{-.5ex} F(\Pi_{N_0} y(t; x, u))\big) \|_{\cH}.
\eea
For the second term on the RHS, we have
\bea 
& \| (\mbox{Id}_{\cH} - \Pi_N) \big(F(y(t; x, u)) - F(\Pi_{N_0} y(t; x, u))\big) \|_{\cH} \\
& \le \|F(y(t; x, u)) - F(\Pi_{N_0} y(t; x, u))\|_{\cH} \\
& \le  \Lip(F\vert_{\mathfrak{B}})\| (\mbox{Id}_{\cH} - \Pi_{N_0} ) y(t; x, u) \|_{\cH}, \qquad  \Forall t\in [0,T], \; u \in \mathcal{U}_{ad},
\eea
where $\mathfrak{B}$ denotes the ball in $\cH$ centered at the origin with radius $ \mathcal{C}$ given by \eqref{Eq_y_uniform-in-u_bounds}.

It follows then from \eqref{Eq_unif_est_ys} that
\be \label{Eq_unif_est_Fs}
\sup_{u\in \mathcal{U}_{ad}} \sup_{t\in[0,T]} \| (\mbox{Id}_{\cH} - \Pi_N) \big(F(y(t; x, u)) - F(\Pi_{N_0} y(t; x, u))\big) \|_{\cH} \le \Lip(F\vert_{\mathfrak{B}}) \; \epsilon.
\ee

Due to {\bf (A6)}, for each $x$ in $\cH$,  $\Pi_{N_0} y(s; x, u)$ lies in a uniformly (in $s$ and $u$) bounded subset  of $\cH_{N_0} $. The finite dimensionality of $\cH_{N_0}$ (as a Galerkin subspace) ensures the compactness of the set 
\bes
\mathcal{E}_x:=\overline{\{\Pi_{N_0} y(s; x, u), \; u \in \mathcal{U}_{ad}, \; s\in [0, T]\}}^{\cH}.
\ees

For each $z$ in $\mathcal{E}_x$ and $\epsilon>0$, due to  \eqref{Eq_identity_approx}  there exists an integer $N_1(z)$ such that 
\bes
\|  (\mbox{Id}_{\cH} - \Pi_N) F(z)\| \leq \frac{\epsilon}{2}, \; \; \Forall N \ge N_1(z).
\ees
Since $F$ is continuous, there exists a neighborhood $\mathcal{N}_z$ of $z$ in $\cH_{N_0}$ such that 
\be\label{eq_control_F}
\|  (\mbox{Id}_{\cH} - \Pi_N) F(w)\| \leq \epsilon, \; \; \Forall N \ge N_1(z), \; w \in \mathcal{N}_z.
\ee
From the compactness of $\mathcal{E}_x$ we can extract a finite cover of $\mathcal{E}_x$ by such neighborhoods $ \mathcal{N}_z$ for which
\eqref{eq_control_F} holds, and thus one can ensure the existence of an integer $N_1$ for which 
\be
\sup_{z \in \mathcal{E}_x} \|  (\mbox{Id}_{\cH} - \Pi_N)   F(z) \|_{\cH} \le \epsilon, \qquad \Forall N \ge N_1.
\ee
This last inequality ensures for each $x$ in $\cH$, the existence of an integer $N_1$ for which
\be \label{Eq_unif_est_Fn}
\sup_{u \in  \mathcal{U}_{ad}} \sup_{t\in[0, T]} \|  (\mbox{Id}_{\cH} - \Pi_N)   F(\Pi_{N_0} y(t;x, u)) \|_{\cH} \le \epsilon, \qquad \Forall N \ge N_1.
\ee
Then, \eqref{Eq_highmode_F_est} follows from \eqref{Eq_unif_est_F}, \eqref{Eq_unif_est_Fs} and \eqref{Eq_unif_est_Fn}. 
\ep

We are now in position to formulate a uniform (in $u$) version of Lemma~\ref{Lem:local_in_u_conv} in which the growth condition 
\eqref{Eq_poly_growth_C} is no longer required.

\bt  \label{Lem:uniform_in_u_conv}

Assume that {\bf (A0)}--{\bf (A3)} and {\bf (A5)}--{\bf (A7)} hold. Assume also that $\mathfrak{C}:V\rightarrow \cH$
is continuous.  Then, for any $(x,u)$ in $\cH \times \mathcal{U}_{ad}$, the mild solution $y(t;x,u)$ to \eqref{ODE}  satisfies
the following uniform convergence result
\be  \label{uniform_in_u_conv_Goal}
\lim_{N\rightarrow \infty}  \sup_{u\in \mathcal{U}_{ad}} \sup_{t \in [0, T]} \|y_N(t; \Pi_N x, u) - y(t; x,u)\|_{\cH} = 0,
\ee
where $y_N$ denotes the solution to the Galerkin approximation \eqref{ODE_Galerkin}.
\et

%%%%%%%Proof of Theorem 2.1 %%%%%%%%%
\bp

Compared to Lemma~\ref{Lem:local_in_u_conv}, we have replaced {\bf (A4)} by the stronger assumption {\bf (A6)} and the set of admissible controls $\mathcal{U}_{ad}$ is as given in {\bf (A5)}. By following the proof of Lemma~\ref{Lem:local_in_u_conv}, in order to obtain the uniform convergence result \eqref{uniform_in_u_conv_Goal}, it suffices to show that the two terms $\int_0^T d_N(s; u) \d s$ and $\int_0^T  \widetilde{d}_N(s; u) \d s$ involved in the RHS of \eqref{eq_rn_gronwall} converge to zero as $N \rightarrow \infty$ uniformly with respect to $u$  in $\mathcal{U}_{ad}$. 

Recall from \eqref{totoc} that  
\bes
d_N(s; u) = \sup_{t\in[s, T]} \| \big( T(t-s) -  e^{L_N (t -s )} \Pi_N \big) F(y(s; x, u)) \|_{\cH},
\ees
 which is defined for every $s$ in $[0, T]$ and $u$ in $\mathcal{U}_{ad}$.
 Thanks to Lemma~\ref{Lem:highmode_F_est}, for any fixed $\epsilon > 0$, there exists $N_0$  in  $\mathbb{Z}_+^\ast$, such that
\be  \label{eq:uniform_est_dn_1}
\|(\mbox{Id}_{\cH} - \Pi_{N_0}) F(y(s;x, u))\|_{\cH} < \epsilon, \qquad  \; s\in [0,T], \; u \in \mathcal{U}_{ad}.
\ee
Now, for $N_0$ chosen above, we have
\bea  \label{eq:uniform_est_dn_2}
&\| \big( T(t-s) -  e^{L_N (t -s )} \Pi_N \big) F(y(s; x, u)) \|_{\cH} \\
&\hspace{4em} \le  \| \big( T(t-s) -  e^{L_N (t -s )} \Pi_N \big) \Pi_{N_0} F(y(s; x, u)) \|_{\cH} \\
& \hspace{7em}  + \| \big( T(t-s) -  e^{L_N (t -s )} \Pi_N \big) (\mbox{Id}_{\cH} - \Pi_{N_0} ) F(y(s; x, u)) \|_{\cH}.
\eea
By using \eqref{Eq_control_T_t} and Assumption {\bf (A1)}, we obtain:
\bea \label{eq:uniform_est_dn_3}
 & \hspace{-3em}  \sup_{t\in[s, T]} \| \big( T(t-s) -  e^{L_N (t -s )} \Pi_N \big) (\mbox{Id}_{\cH} - \Pi_{N_0}) F(y(s; x, u)) \|_{\cH} \\
& \hspace{7em}\le  2 Me^{\omega T}\|(\mbox{Id}_{\cH} - \Pi_{N_0}) F(y(s; x, u)) \|_{\cH} \\
 & \hspace{7em} \le 2 Me^{\omega T}\epsilon,  \quad  \Forall s\in [0,T], \; u \in \mathcal{U}_{ad},
\eea
where the last inequality follows from \eqref{eq:uniform_est_dn_1}.

Due to {\bf (A6)}, for each $x$ in $\cH$,  $y(s; x, u)$ lies in a uniformly (in $s$ and $u$) bounded subset  of $\cH$. This  together with the locally Lipschitz property of $F$ implies that $\Pi_{N_0} F(y(s; x, u))$ lies in a bounded subset  of $\cH_{N_0} $ for all $u$  in $\mathcal{U}_{ad}$ and for all $s$ in $[0, T]$. The finite dimensionality of $\cH_{N_0}$ ensures then the compactness of the set 
\bes
\mathcal{E}'_x=\overline{\{\Pi_{N_0} F(y(s; x, u)), \; u \in \mathcal{U}_{ad}, \; s\in [0, T]\}}^{\cH}.
\ees

Now, for each $z$ in $\mathcal{E}'_x$ and $\epsilon>0$, the convergence property \eqref{eq:linear_pointwise_convergence} valid uniformly over bounded time-intervals allows us to ensure the existence of an integer $N_1(z)$, for which
\bes
\sup_{t\in[0, T]} \| \big( T(t) -  e^{L_N t} \Pi_N \big)  z \|_{\cH} \le \frac{\epsilon}{2}, \qquad \Forall \; N \ge N_1(z).
\ees
Then, by using \eqref{Eq_control_T_t} and Assumption {\bf (A1)}, there exists a neighborhood $\mathcal{N}_z$ of $z$ in $\cH_{N_0}$ such that 
\be \label{eq_control_F_v2}
\sup_{t\in[0, T]} \| \big( T(t) -  e^{L_N t} \Pi_N \big)  w \|_{\cH} \le \epsilon, \qquad \Forall \; N \ge N_1(z), \; w \in \mathcal{N}_z.
\ee
From the compactness of $\mathcal{E}'_x$ we can extract a finite cover of $\mathcal{E}'_x$ by such neighborhoods in which
\eqref{eq_control_F_v2} holds, and thus one can ensure the existence of an integer $N_1$ for which 
\be \label{eq_control_F_v2b}
\sup_{t\in[0, T]} \| \big( T(t) -  e^{L_N (t)} \Pi_N \big) w \|_{\cH} \le \epsilon, \qquad \Forall  N \ge N_1, \; w \in \mathcal{E}'_x.
\ee
Now, for each fixed $s$ in $[0, T]$ and $u$ in $\mathcal{U}_{ad}$, by taking $w = \Pi_{N_0} F(y(s; x, u))$, we get from 
\eqref{eq_control_F_v2b} that
\bes
\sup_{t\in[0, T]} \| \big( T(t) -  e^{L_N (t)} \Pi_N \big) \Pi_{N_0} F(y(s; x, u)) \|_{\cH} \le \epsilon, \qquad \Forall  N \ge N_1.
\ees
It follows then for  all $u$  in $\mathcal{U}_{ad}$
\be \label{eq:uniform_est_dn_4}
\hspace{-2ex}\sup_{t\in[s, T]} \| \big( T(t-s) -  e^{L_N (t -s )} \Pi_N \big)  \Pi_{N_0} F(y(s;x, u)) \|_{\cH} \le \epsilon, \;\Forall N \ge N_1, \;  s\in [0,T].
\ee

By using \eqref{eq:uniform_est_dn_3} and \eqref{eq:uniform_est_dn_4} in \eqref{eq:uniform_est_dn_2}, we get
\bes
d_N(s; u) \le (1 + 2 Me^{\omega T})\epsilon, \quad  \Forall N \ge N_1, \; s\in [0,T], \; u \in \mathcal{U}_{ad},
\ees
which leads to 
\be \label{eq:uniform_est_dn_vf}
 \sup_{u \in \mathcal{U}_{ad}} \int_0^T d_N(s; u) \d s  \le (1 + 2 Me^{\omega T}) T \epsilon, \quad  \Forall N \ge N_1.
\ee

We consider now the term $\sup_{u \in \mathcal{U}_{ad}} \int_0^T  \widetilde{d}_N(s; u) \d s$ with 
\bes
\widetilde{d}_N (s; u) = \sup_{t\in[s, T]} \| \big( T(t-s) -  e^{L_N (t -s )} \Pi_N \big) \mathfrak{C}(u(s)) \|_{\cH}
\ees 
as defined in \eqref{totod} for almost every $s$ in $[0,T]$ and every $u$ in $\mathcal{U}_{ad}$ here. 

Since the set $U \subset V$ is compact (cf.~Assumption {\bf (A5)}) and $\mathfrak{C}: V \rightarrow \cH$ is continuous, then  $\mathfrak{C}(U)$ is a compact set of $\cH$. Following a compactness argument similar to that used to derive \eqref{eq_control_F_v2b}, we can ensure the existence of an integer $N_2$ such that
\be \label{eq_control_C}
\sup_{t\in[0, T]} \|\big( T(t) -  e^{L_N (t)} \Pi_N \big) \mathfrak{C}(w)\|_{\cH} \le \epsilon, \qquad \Forall N \ge N_2, \; w \in U.
\ee
Now, for each $u$ in $\mathcal{U}_{ad}$, since $u(s)$ takes value in $U$ for almost every $s$ in $[0,T]$, we obtain from \eqref{eq_control_C} that
\be
\sup_{t\in[s, T]} \|\big( T(t-s) -  e^{L_N (t -s )} \Pi_N \big) \mathfrak{C}(u(s))\|_{\cH} \le \epsilon, \quad \text{for all} \; N \ge N_2 \text{ and a.e. } s \in [0,T].
\ee
It follows then that for all $N \geq N_2$, and $u$ in $\mathcal{U}_{ad}$,
\be  \label{eq:uniform_u_est}
\int_0^T  \widetilde{d}_N(s; u) \d s  =  \int_0^T \sup_{t\in[s, T]} \| \big( T(t-s) -  e^{L_N (t -s )} \Pi_N \big) \mathfrak{C}(u(s)) \|_{\cH} \d s \le T \epsilon,
\ee
and thus,
\be   \label{eq:uniform_est_tilde_dn_vf}
\sup_{u \in \mathcal{U}_{ad}} \int_0^T  \widetilde{d}_N(s; u) \d s \le T \epsilon,  \quad \Forall  N \ge N_2.
\ee

The desired uniform convergence of $\int_0^T d_N(s; u) \d s$ and $\int_0^T  \widetilde{d}_N(s; u) \d s$ follows from \eqref{eq:uniform_est_dn_vf} and \eqref{eq:uniform_est_tilde_dn_vf}. The proof is now complete.
\ep
%%%%%%%%%%%%%%%%%%%%%%
%%%%%%%%%%%%%%%%%%%%%%
\br\label{Rmk_F=0}
From the proof given above, it is clear that Assumption {\bf (A5)} is made to ensure the uniform convergence of $\int_0^T  \widetilde{d}_N(s; u) \d s$, while  Assumptions {\bf (A6)} and {\bf (A7)} are made to ensure the uniform convergence of $\int_0^T d_N(s; u) \d s$. The last two assumptions are thus not needed if the nonlinear term $F$ is identically zero.
\er
%%%%%%%%%%%%%%%%%%%%%%

Note that one can readily check that the convergence result stated in Theorem~\ref{Lem:uniform_in_u_conv} also holds when \eqref{ODE} is initialized at any other time instance $t$ in $[0,T)$. This will be needed in the next subsection to derive approximation results for value functions associated with optimal control problems for the IVP \eqref{ODE}.

More precisely, for each $(t,x)$  in $[0,T) \times \cH$, we consider the following evolution problem 
\bea \label{ODE_t_sec3}
\frac{\d y}{\d s} &= L y + F(y) + \mathfrak{C} (u(s)),  \quad s \in (t, T], \; u\in \mathcal{U}_{ad}[t,T],\\
y(t) &=x \, \in \cH,
\eea
with
\be\label{U_ad_toto}
\mathcal{U}_{ad}[t,T]:=\{u\vert_{[t,T]} \; : \; u \in \mathcal{U}_{ad}\},
\ee
and the corresponding Galerkin approximation:
\bea \label{ODE_Galerkin_t}
\frac{\d y_N}{\d s} &= L_N y_N + \Pi_N F(y_N) + \Pi_N \mathfrak{C} (u(s)), \qquad s \in (t, T], \; u\in \mathcal{U}_{ad}[t,T],\\
y_N(t) &= x_N, \qquad x_N := \Pi_N x \in \cH_N.
\eea

Hereafter, we denote by $y_{t,x}(\cdot;u)$ the solution to \eqref{ODE_t_sec3} emanating from $x$ at time $t$, and by $y^N_{t,x_N}(\cdot;u)$ the solution to \eqref{ODE_Galerkin_t} emanating from $x_N:=\Pi_N x$ at time $t$.
We have then the following corollary of Theorem~\ref{Lem:uniform_in_u_conv}:

\bc \label{Lem:uniform_conv_locally_Lip_t_v2}
Assume that the conditions of Theorem~\ref{Lem:uniform_in_u_conv} hold. Then, for any  mild solution $y_{t,x}(\cdot;u)$ to \eqref{ODE_t_sec3} over $[t,T]$, with $x$  in $\cH$ and $u$ in $\mathcal{U}_{ad}[t,T]$ given in \eqref{U_ad_toto},  the following convergence result is satisfied:
\be \label{uniform_conv_Goal_t_v2}
\lim_{N\rightarrow \infty} \sup_{t\in[0,T]} \sup_{u\in\mathcal{U}_{ad}[t,T]} \sup_{s \in [t, T]} \| y^N_{t,x_N}(s; u) - y_{t,x}(s;u)\|_{\cH} = 0,
\ee
where $ y^N_{t,x_N}$ denotes the solution to the Galerkin approximation \eqref{ODE_Galerkin_t}.
\ec

\bp
Since both the linear operator $L$ and the nonlinearity $F$ in \eqref{ODE_t_sec3} are time independent, then the supremum of $\sup_{u\in\mathcal{U}_{ad}[t,T]} \sup_{s \in [t, T]} \| y^N_{t,x_N}(s; u) - y_{t,x}(s;u)\|_{\cH}$ as $t$  varies in $[0,T]$, is achieved at $t=0$, i.e.,
\bea
\hspace{-.3ex} \sup_{u\in\mathcal{U}_{ad}[t,T]} \underset{\substack{t\in[0,T]\\{ s \in [t, T]}}}\sup \| y^N_{t,x_N}(s; u) \hspace{-.3ex}- \hspace{-.3ex}y_{t,x}(s;u)\|_{\cH} &= \hspace{-2.3ex}\sup_{u\in\mathcal{U}_{ad}[0,T]} \sup_{s \in [0, T]} \| y^N_{0,x_N}(s; u) \hspace{-.3ex}- \hspace{-.3ex}y_{0,x}(s;u)\|_{\cH} \\
& \hspace{-5ex}=  \hspace{-2.3ex}\sup_{u\in\mathcal{U}_{ad}[0,T]} \sup_{s \in [0, T]} \| y_N(s;x_N, u) \hspace{-.3ex}-\hspace{-.3ex} y(s;x,u)\|_{\cH}.
\eea
The desired estimate \eqref{uniform_conv_Goal_t_v2} follows then from \eqref{uniform_in_u_conv_Goal}, noting that $\mathcal{U}_{ad}[0,T]=\mathcal{U}_{ad}$.
\ep

\subsection{Galerkin approximations of optimal control and value functions: Convergence results}\label{Sec_cve}

We assume in this section  that $U$ is a compact and convex subset of  the Hilbert space
$V$. In particular this ensures that $\mathcal{U}_{ad}$ defined in \eqref{Eq_U_bounded} is a bounded, closed and convex set. 
For such an admissible set of controls, conditions of existence to optimal control problems associated with the IVPs \eqref{ODE} and \eqref{ODE_Galerkin} are recalled in Appendix \ref{Sec_existence_opt_contr}.

We introduce next the cost functional, $J\colon \cH \times \mathcal{U}_{ad}  \rightarrow \mathbb{R}^+$, associated with the IVP \eqref{ODE}:
\be  \label{J_sec3}
J(x,u) := \int_0^T [\mathcal{G}(y(s; x, u)) + \mathcal{E}(u(s)) ] \, \d s,  \; x \in \cH,
\ee 
where $\mathcal{G}: \cH \rightarrow \mathbb{R}^+$ and $\mathcal{E}: V \rightarrow \mathbb{R}^+$ are assumed to be {\it continuous}, and $\mathcal{G}$ is assumed to satisfy furthermore the condition:
\be\label{C1}\tag{{\bf C1}}
\mathcal{G}  \mbox{ is  locally Lipschitz in the sense of } \eqref{Local_Lip_cond}.
\ee

The associated optimal control problem then writes 
\be  \label{P_sec3}  \tag {$\mathcal{P}$}
\begin{aligned}
\hspace{-.5ex}\min \, J(x,u)  \hspace{1ex}\text{ s.t. } \hspace{1ex} (y, u) \in L^2(0,T; \cH) \times  \mathcal{U}_{ad} \text{ solves} ~\eqref{ODE} \text{ with }  \; y(0)  = x \in \cH.
\end{aligned}
\ee

The cost functional, $J_N\colon \cH_N \times \mathcal{U}_{ad}  \rightarrow \mathbb{R}^+$, associated with the Galerkin approximation  \eqref{ODE_Galerkin} is given by 
\be  \label{J_Galerkin}
J_N(\Pi_N x,u) := \int_0^T [\mathcal{G}(y_N(s; \Pi_N x, u)) + \mathcal{E}(u(s)) ] \, \d s,  \; x \in \cH,
\ee 
and the corresponding optimal control problem reads: 
\be  \label{P_Galerkin}  \tag {$\mathcal{P}_N$}
\begin{aligned}
& \min \, J_N(\Pi_N x,u)  \quad \text{ s.t. } \quad (y_N, u) \in L^2(0,T; \cH_N) \times  \mathcal{U}_{ad} \text{ solves} ~\eqref{ODE_Galerkin}\\
& \hspace{15em} \text{ with }  \; y_N(0)  = \Pi_N x \in \cH_N.
\end{aligned}
\ee

We assume hereafter that both problems, \eqref{P_sec3} and  \eqref{P_Galerkin}, possess each a solution.
We analyze the convergence of the corresponding value functions by adopting a dynamic programming approach. 
For that purpose, we consider for each $(t,x)$  in $[0,T) \times \cH$ a family of optimal control problems associated with \eqref{ODE_t_sec3} and the following cost functional $J_{t,x}$:
\be  \label{J_tx}
J_{t,x}(u) := \int_t^T [\mathcal{G}(y_{t,x}(s;u)) + \mathcal{E}(u(s)) ] \d s, \qquad t \in [0,T), \; u \in \mathcal{U}_{ad}[t,T].
\ee

The cost functional associated with the corresponding Galerkin approximation \eqref{ODE_Galerkin_t} is given by
\be\label{JN_tx}
J^N_{t,x_N}(u) := \int_t^T [\mathcal{G}(y^N_{t,x_N}(s; u)) + \mathcal{E}(u(s)) ] \d s,  \qquad t \in [0,T), \; u \in \mathcal{U}_{ad}[t,T],
\ee
in which we have denoted $\Pi_N x$ by $x_N.$

The value functions corresponding to the optimal control problems associated respectively with \eqref{ODE_t_sec3} and with \eqref{ODE_Galerkin_t}, are then defined as follows:
\begin{subequations}\label{Eq_val_fcts_sec3}
\begin{align}
& v(t, x) := \inf_{u \in \mathcal{U}_{ad}[t,T]} J_{t,x}(u),  \; \; \forall \; (t,x) \in [0,T) \times \cH \; \;  \text{ and } \;\;  v(T, x) = 0, \label{subEq1}\\
& v_N(t, x_N) := \inf_{u \in \mathcal{U}_{ad}[t,T]} J^N_{t,x_N}(u),  \;\;  \forall \; (t,x_N) \in [0,T)\times \cH_N  \;\;   \text{ and }  \;  v_N(T, x_N) = 0. \label{subEq2}
\end{align}
\end{subequations}

We have then the following result.
\bt \label{Thm_cve_Galerkin_val}
Assume that the conditions in Theorem~\ref{Lem:uniform_in_u_conv} together with \eqref{C1} hold. Furthermore, let there exists for each pair $(t,x)$ a  minimizer $u_{t,x}^*$ (resp.~$u_{t,x}^{N,*}$) in $\mathcal{U}_{ad}[t,T]$ of the minimization problem in \eqref{subEq1} (resp.~in \eqref{subEq2}).

 Then for any $x$ in $\cH$, it holds that
\be \label{value_est_goal}
\lim_{N \rightarrow \infty} \sup_{t \in [0, T]} |v_N(t,\Pi_N x) - v(t,x)| = 0.
\ee

\et

\bp

By the definition of the value functions in \eqref{Eq_val_fcts_sec3}, we have
\be \label{value_est_1}
v(t, x) = J_{t,x}(u_{t,x}^*) \le  J_{t,x}(u_{t,x}^{N,*}), 
\ee
and
\be
v_N(t, x_N) = J^N_{t,x_N}(u_{t,x}^{N,*}), 
\ee
with $x_N :=\Pi_N x$. 

The inequality \eqref{value_est_1} and the definition of $J$ give then
\be
v(t, x) \le \int_t^T [\mathcal{G}(y_{t,x}(s;u_{t,x}^{N,*})) + \mathcal{E}(u_{t,x}^{N,*}(s)) ] \d s.
\ee

By subtracting $J^N_{t,x_N}(u_{t,x}^{N,*})$ on both sides of the above inequality, we get
\bea \label{value_est_3}
\hspace{-1.8ex}v(t, x)\hspace{-.3ex} - \hspace{-.3ex} v_N(t, x_N)  & \hspace{-.3ex}\le \hspace{-.3ex}\int_t^T \hspace{-.5ex} [\mathcal{G}(y_{t,x}(s;u_{t,x}^{N,*})) \\
&\hspace{1em}+ \mathcal{E}(u_{t,x}^{N,*}(s)) ] \d s - \hspace{-.5ex} \int_t^T \hspace{-.5ex} [\mathcal{G}(y^N_{t,x_N}(s; u_{t,x}^{N,*})) + \mathcal{E}(u_{t,x}^{N,*}(s)) ] \d s \\
& = \int_t^T \hspace{-.5ex} [\mathcal{G}(y_{t,x}(s;u_{t,x}^{N,*})) - \mathcal{G}(y^N_{t,x_N}(s; u_{t,x}^{N,*}))] \d s.
\eea
Besides, since both $L$ and $F$ are time-independent, it follows from {\bf (A6)} that there exists a positive constant $\mathcal{C}$ such that
\bea
& \|y_{t,x}(s; u_{t,x}^{N,*})\| \le \mathcal{C},  && \Forall t \in [0, T),\; s \in [t, T], \; n \in \mathbb{Z}_+^\ast,\\
& \|y^N_{t,x_N}(s; u_{t,x}^{N,*})\| \le \mathcal{C}, && \Forall t \in [0, T), \; s \in [t, T], \; n \in \mathbb{Z}_+^\ast.
\eea
Now by denoting by $\mathfrak{B}$ the ball in $\cH$ with radius $\mathcal{C}$ centered at the origin, we have
\bea  \label{value_est_4}
 & \int_t^T [\mathcal{G}(y_{t,x}(s;u_{t,x}^{N,*})) - \mathcal{G}(y^N_{t,x_N}(s; u_{t,x}^{N,*}))] \d s \\
  & \le   \Lip(\mathcal{G}\vert_{\mathfrak{B}}) \int_t^T \|y_{t,x}(s;u_{t,x}^{N,*}) - y^N_{t,x_N}(s; u_{t,x}^{N,*})\|_{\cH} \d s \\
& \le   T \Lip(\mathcal{G}\vert_{\mathfrak{B}}) \sup_{s\in[t, T]}\|y_{t,x}(s;u_{t,x}^{N,*}) - y^N_{t,x_N}(s;u_{t,x}^{N,*})\|_{\cH},
\eea
which together with \eqref{value_est_3} leads to
\be  \label{value_est_5}
v(t, x) - v_N(t, x_N)  \le T \Lip(\mathcal{G}\vert_{\mathfrak{B}})\sup_{s\in[t, T]}\|y_{t,x}(s; u_{t,x}^{N,*}) - y^N_{t,x_N}(s; u_{t,x}^{N,*})\|_{\cH}.
\ee
Similarly, we have
\be \label{value_est_6}
v_N(t, x_N)  - v(t, x)  \le T  \Lip(\mathcal{G}\vert_{\mathfrak{B}})  \sup_{s\in[t, T]} \|y_{t,x}(s;u_{t,x}^*) - y^N_{t,x_N}(s; u_{t,x}^*)\|_{\cH}.
\ee
The convergence result \eqref{value_est_goal} follows then from \eqref{value_est_5}, \eqref{value_est_6}, and Corollary~\ref{Lem:uniform_conv_locally_Lip_t_v2}.
\ep

\br \label{Rmk:time-dependent-F}
If the nonlinearity $F$ in the IVP \eqref{ODE} depends also on time, by modifying Assumption {\bf (A3)} accordingly, all the results of Section~\ref{Sect_Galerkin} still hold literally except that of Lemma~\ref{Lem:highmode_F_est} that needs to be amended. For instance, by replacing {\bf (A3)} by the following assumption 
\bi
\item[{\bf (A3$'$)}] The nonlinearity $F: [0,T] \times \cH \rightarrow \cH$ satisfies that $F(\cdot,y) \in L^\infty(0,T; \cH)$ for every $y$ in $\cH$,  $F(t,\cdot)$ is locally Lipschitz for almost every $t$ in $[0, T]$, and for any given bounded set $\mathfrak{B} \subset \cH$, the mapping $t \mapsto \Lip(F(t, \cdot)\vert_\mathfrak{B})$ is in $L^\infty(0,T)$, where $\Lip(F(t, \cdot)\vert_\mathfrak{B})$ denotes the Lipschitz constant of $F(t,\cdot)$ on the set $\mathfrak{B}$. 
\ei

Under this new assumption, by replacing in Lemma~\ref{Lem:highmode_F_est} the supremum for $t$ over $[0,T]$ in \eqref{Eq_highmode_F_est} with the essential supremum,  we have
\be 
\lim_{N\rightarrow \infty} \sup_{u\in \mathcal{U}_{ad}}  \mathop{\mathrm{ess}\, \sup}_{t\in[0,T]} \|(\mathrm{Id}_{\cH} - \Pi_N) F(t, y(t;x,u))\|_{\cH} =0.
\ee
\er

%%%%%%%%%%%%%%%%%%%%%%%%%%%%%%%%%%%%%%%%%%
%%%%%%%%%%%%%%%%%%%%%%%%%%%%%%%%%%%%%%%%%%
\subsection{Galerkin approximations of optimal control and value functions: Error estimates}\label{Sec_Err_estimates}
We provide in this section some simple and useful error estimates in terms of their interpretations. 

For that purpose, we assume throughout this subsection the following set of assumptions collected as follows 
\bi

\item[{\bf(E)}]  
\bi

\item The linear operator $L: D(L) \subset \cH \rightarrow \cH$ is self-adjoint. 
\item The Galerkin approximations \eqref{ODE_Galerkin} are constructed based on the eigen-subspaces $\cH_{N}  := \mathrm{span}\{e_k \;:\; k = 1, \cdots, N\}, \; N\in \mathbb{Z}_+^\ast$,  where the $e_k$'s are the eigenfunctions of $L$.

\item Assumption {\bf (A6)}. 
\item  Assumption \eqref{C1}.
\ei

\ei
We take the set of admissible controls, $\mathcal{U}_{ad}$, to be given by \eqref{Eq_U_bounded}. But in contrast to Sect.~\ref{Sec_cve}, the set $U$ in the definition of $\mathcal{U}_{ad}$ is not assumed to be compact in $V$.

Hereafter within this subsection, $\mathfrak{B}$ denotes the ball in $\cH$ centered at the origin with radius $\mathcal{C}$, where  $\mathcal{C}$ is the same as given in Assumption {\bf (A6)}. We start with a basic pointwise estimate between  the cost functional $J_{t, x}$  given by \eqref{J_tx}  and its approximation $J_{t,x_{N}}^N$ given by \eqref{JN_tx}.

\bl \label{lem:J_estimates}
Under the set of assumptions given by {\bf (E)}, for any  $(t,x) \in [0,T)\times \mathcal{H}$, and $u\in \mathcal{U}_{ad}[t,T]$, there exists $\gamma >0$, independent of $N$, such that for all $u$ in $ \mathcal{U}_{ad}[t,T],$
\be  \label{Eq_J_est}      
|J_{t, x}(u) -  J_{t,x_{N}}^N(u)|  \le \Lip(\mathcal{G} \vert_{\mathfrak{B}}) \left[\sqrt{T-t} + \gamma (T-t) \right] \| \Pi_N^\perp y_{t,x}(\cdot;u)\|_{L^2(t,T; \cH)}.
\ee
\el

\bp

Note that by the definitions of $J_{t, x}$ and $J_{t,x_{N}}^N$ and the locally Lipschitz condition \eqref{C1} on $\mathcal{G}$, we have
\be
|J_{t, x}(u) -  J_{t,x_{N}}^N(u)| \le \Lip(\mathcal{G}\vert_{\mathfrak{B}}) \int_{t}^T \|y_{t,x}(s;u) - y^N_{t,x_N}(s;u)\|_{\cH} \d s, 
\ee
where $\mathfrak{B} \subset \cH$ denotes a ball centered at the origin that contains $y_{t,x}(s;u)$ and $y^N_{t,x_N}(s;u)$ for all $s$ in $[t,T]$. By rewriting $y_{t,x}(s;u)$ as $\Pi_{N}y_{t,x}(s;u) + \Pi_{N}^\perp y_{t,x}(s;u)$, we have thus 
\bea \label{J_est1}
&\hspace{-2ex}|J_{t, x}(u) -  J_{t,x_{N}}^N(u)| \\
&\hspace{2em}\le \Lip(\mathcal{G}\vert_{\mathfrak{B}}) \int_{t}^T \hspace{-1ex}  \Big( \| \Pi_{N}y_{t,x}(s;u) - y^N_{t,x_N}(s;u)\|_{\cH} +  \| \Pi_{N}^\perp y_{t,x}(s;u) \|_{\cH}  \Big) \d s.
\eea

The estimate \eqref{Eq_J_est} follows then immediately if one proves that there exists $\gamma^2 > 0$ independent of $N$ such that for any  $(t,x) \in [0,T)\times \mathcal{H}$, and $u\in \mathcal{U}_{ad}[t,T]$,
\be  \label{low_mode_est}
\|\Pi_{N}y_{t,x}(s;u) -y^N_{t,x_{N}}(s;u) \|^2_{\cH} \le  \gamma^2 \int_t^s \| \Pi_N^\perp y_{t,x}(s';u) \|^2_{\cH} \, \d s', \qquad  s \in [t, T].
\ee

Indeed, in such a case we get 
\bea
 \hspace{-3ex}\int_{t}^T  \hspace{-.75ex} \| \Pi_{N}y_{t,x}(s;u) - y^N_{t,x_N}(s;u)\|_{\cH}  \d s & \le \gamma  \int_{t}^T   \hspace{-1.3ex}\Big(\int_t^s \| \Pi_N^\perp y_{t,x}(s';u) \|^2_{\cH} \, \d s' \Big)^{\frac{1}{2}}\d s \\
& \le \gamma  (T-t) \| \Pi_N^\perp y_{t,x}(\cdot;u)\|_{L^2(t,T; \cH)},
\eea
and by noting from  H\"older's inequality that 
\bes
 \int_{t}^T   \| \Pi_{N}^\perp y_{t,x}(s;u) \|_{\cH}  \d s \le \sqrt{T-t}  \| \Pi_N^\perp y_{t,x}(\cdot;u)\|_{L^2(t,T; \cH)},
\ees
we arrive at \eqref{Eq_J_est}.

We are thus left with the proof of  \eqref{low_mode_est} which is easily derived as follows. Let us introduce 
\be\label{defw}
w(s):= \Pi_{N}y_{t,x}(s;u) -y^N_{t,x_{N}}(s;u). 
\ee 
By applying $\Pi_N$ to both sides of Eq.~\eqref{ODE_t_sec3}, we obtain that $\Pi_{N}y_{t,x}(\cdot; u)$ satisfies the following IVP:
\beas
\frac{\d \Pi_{N}y}{\d s} &= L_N \Pi_N y + \Pi_{N}F (\Pi_N y + \Pi_{N}^\perp y_{t,x}(s;u)) +  \Pi_{N}\mathfrak{C} (u(s)),  \quad s \in (t,T], \\
\Pi_N y(t)  &= \Pi_{N} x \in \cH_N.
\eeas
This together with \eqref{ODE_Galerkin_t} implies that $w$ satisfies the following problem:
\bea \label{eq:w}
\frac{\d w}{\d s} &= L_{N} w + \Pi_{N}  F(\Pi_N y + \Pi_{N}^\perp y_{t,x}(s;u)) -  \Pi_{N}F(y_N),  \quad s \in (t,T], \\
w(t)  &= 0.
\eea

By taking the $\cH$-inner product on both sides of \eqref{eq:w} with $w$, we obtain:
\be  \label{energy est:1}
\frac{1}{2}\frac{\d \|w\|^2_{\cH}}{\d s} = \langle L_{N} w, w \rangle + \langle \Pi_{N} \bigl( F(\Pi_N y + \Pi_{N}^\perp y_{t,x}(s;u)) -  F(y_N)\bigr), w \rangle.
\ee

The local Lipschitz property of $F$ implies then that 
\bea  \label{RHS_control}
\langle \Pi_{N} \bigl( F(\Pi_N y &+ \Pi_{N}^\perp y_{t,x}(s;u)) -  F(y_N)\bigr), w \rangle\\
 & \le  \mathrm{Lip}(F|_{\mathfrak{B}}) \, \| \Pi_N y + \Pi_{N}^\perp y_{t,x}(s;u) - y_N \|_{\cH} \, \|w\|_{\cH} \\
 &\le  \mathrm{Lip}(F|_{\mathfrak{B}}) \, (\| w\| + \| \Pi_{N}^\perp y \|_{\cH}) \, \|w\|_{\cH} \\
 &\le  \mathrm{Lip}(F|_{\mathfrak{B}}) \Big(\frac{3}{2} \| w\|_{\cH}^2  + \frac{1}{2} \| \Pi_{N}^\perp y  \|_{\cH}^2\Big). 
\eea

Since $L$ is self-adjoint, we have 
\bea \label{RHS_control:3}
\langle L_{N} w(s), w(s) \rangle = \sum_{i=1}^N  \beta_i \|w_i(s)\|_{\cH}^2\le \beta_1 \|w(s)\|_{\cH}^2,
\eea
where $\beta_i$ is the eigenvalue associated with its  $i^{\mathrm{th}}$ eigenmode $e_i$.

Using \eqref{RHS_control} and \eqref{RHS_control:3} in \eqref{energy est:1}, we finally arrive at 
\be\label{Eq_interm1}
\frac{1}{2}\frac{\d \|w(s)\|_{\cH}^2}{\d s} \le \left(\beta_1 + \frac{3}{2} \mathrm{Lip}(F|_{\mathfrak{B}})   \right) \|w(s)\|_{\cH}^2 + \frac{1}{2} \mathrm{Lip}(F|_{\mathfrak{B}})  \| \Pi_N^\perp y_{t,x}(s;u)\|_{\cH}^2,
\ee
which for all $s$ in $[t, T]$, by a standard application of Gronwall's inequality, leads to 
\bea \label{y-z est}
&\hspace{-1em} \|\Pi_{N}y_{t,x}(s;u) -y^N_{t,x_{N}}(s;u)\|_{\cH}^2  = \|w(s)\|_{\cH}^2   \\
&  \hspace{7em}\le \mathrm{Lip}(F|_{\mathfrak{B}}) \int_t^s e^{2[\beta_1 + \frac{3}{2}\mathrm{Lip}(F|_{\mathfrak{B}})](s-s')}  \| \Pi_N^\perp y_{t,x}(s';u)\|_{\cH}^2 \d s' \\
&  \hspace{7em}\le e^{2[\beta_1 + \frac{3}{2}\mathrm{Lip}(F|_{\mathfrak{B}})]T} \mathrm{Lip}(F|_{\mathfrak{B}}) \int_t^s \| \Pi_N^\perp y_{t,x}(s';u)\|_{\cH}^2 \d s',
\eea
taking into account that $w(t)= 0$ due to \eqref{defw}. The estimate \eqref{low_mode_est} is thus verified, and the proof is complete.  
\ep

We have then
\bt\label{Thm_PM_val}
Assume the set of assumptions given by {\bf (E)}. Assume also that for each $(t,x) \in [0,T)\times \cH$, there exists a minimizer $u^*_{t,x}$ (resp.~$u^{N,\ast}_{t,x_{N}}$) for the value function $v$ (resp.~$v_N$) defined in \eqref{Eq_val_fcts_sec3}.

Then for any $(t,x) \in [0,T)\times \cH$ it holds that
\bea\label{PM_value_est_goal}
&  |v(t,x) - v_N(t,x_{N})| \\
 & \le \Lip(\mathcal{G}\vert_{\mathfrak{B}}) \left[\sqrt{T-t} + \gamma (T-t) \right] \Bigl( \|\Pi_N^\perp y_{t,x}(\cdot; u^\ast_{t,x})\|_{L^2(t,T; \cH)} \\
 &\hspace{17em}+ \| \Pi_N^\perp y_{t,x}(\cdot; u^{N,\ast}_{t,x_{N}})\|_{L^2(t,T; \cH)} \Bigr),
\eea
where the constant $\gamma$ is the same as in Lemma~\ref{lem:J_estimates}.
 
\et

\bp

The result is a direct consequence of Lemma~\ref{lem:J_estimates}. Indeed, since
\bes 
v(t,x)  = J_{t,x}(u^*_{t,x}) \le J_{t,x}(u^{N,\ast}_{t,x_{N}}) \quad \text{and} \quad  v_N(t,x_{N}) = J^N_{t,x_{N}}(u^{N,\ast}_{t,x_{N}}), 
\ees
we get
\bes 
 v(t,x) - v_N(t,x_{N})  \le J_{t,x}(u^{N,\ast}_{t,x_{N}}) - J^N_{t,x_{N}}(u^{N,\ast}_{t,x_{N}}). 
 \ees
It follows then from \eqref{Eq_J_est} that
\bea \label{PM_val_est1}
 \hspace{-1em}v(t,x) - v_N(t,x_{N})\hspace{-.3ex} & \hspace{-.5ex}\le \hspace{-.5ex} \Lip(\mathcal{G} \vert_{\mathfrak{B}}) \left[{\sqrt{T-t}}+ \gamma (T-t) \right] 
 \| \Pi_N^\perp y_{t,x}(\cdot;u^{N,\ast}_{t,x_{N}})\|_{L^2(t,T; \cH)}.
\eea

Similarly, 
\bea \label{PM_val_est2}
\hspace{-1em} v_N(t,x_{N})  -  v(t,x) &\hspace{-.5ex} \le \hspace{-.5ex} \Lip(\mathcal{G}\vert_{\mathfrak{B}}) \left[{\sqrt{T-t}}+ \gamma (T-t) \right] 
 \|\Pi_N^\perp y_{t,x}(\cdot; u^\ast_{t,x})\|_{L^2(t,T; \cH)}.
\eea
The estimate \eqref{PM_value_est_goal} results then from \eqref{PM_val_est1} and \eqref{PM_val_est2}.
\ep

\br \label{Rmk_motivation_A7}
Note that in the RHS of \eqref{PM_value_est_goal},  it is not clear a priori that 
\be \label{unif_vanishing}
\lim_{N\rightarrow \infty} \|\Pi_N^\perp y_{t,x}(\cdot; u^{N,\ast}_{t,x_{N}})\|_{L^2(t,T; \cH)} = 0. 
\ee  
The reason relies on the dependence on $u^{N,\ast}_{t,x_{N}}$ of $\| \Pi_N^\perp y_{t,x}(\cdot; u^{N,\ast}_{t,x_{N}})\|_{L^2(t,T; \cH)}$, where $u^{N,\ast}_{t,x_{N}}$ denotes the control synthesized from the $N$-dimensional Galerkin approximation.

Estimate  \eqref{PM_value_est_goal} provides thus another perspective regarding the usage of Assumption {\bf (A7)} to ensure the convergence of the value functions associated with the Galerkin approximations presented in Sect.~\ref{Sec_cve}. Note that Assumption {\bf (A7)} is just a slightly strengthened form of \eqref{unif_vanishing}. See Sect.~\ref{Sect_examples} for a broad class of IVPs for which Assumption {\bf (A7)} is satisfied. 

\er

\bc\label{Lem_controller_est}
Assume that the conditions given in Theorem~\ref{Thm_PM_val} hold.
Let us also denote $u^* := u^*_{0,x}$ and $u^{*}_{N}:= u^{N,*}_{0,x_N}$. Assume furthermore that there exists $\sigma >0$ such that the following local growth condition is satisfied  for the cost functional $J$ defined in \eqref{J_sec3}: 
\be\label{Eq_growth_onJ}
\sigma \|u^*  - v\|_{L^q(0,T; V)}^q \le J(x, v) - J(x, u^*), 
\ee
for all $v$ in some neighborhood $\mathcal{W} \subset \mathcal{U}_{ad}$ of $u^*$, with $\mathcal{U}_{ad}$ given by \eqref{Eq_U_bounded}. Assume finally that $u^*_N$  lies in  $\mathcal{W}$. 
Then, 
\bea\label{Est_contr_diff}
\|u^\ast - u^\ast_{N}\|_{L^q(0,T; V)}^q & \le \frac{1}{\sigma}\Lip(\mathcal{G}\vert_{\mathfrak{B}}) \left[\sqrt{T} + \gamma T  \right] \Bigl( \| \Pi_N^\perp y(\cdot;u^*)\|_{L^2(0,T; \cH)}\\ 
& \hspace{10em}+  2 \|  \Pi_N^\perp y (\cdot; u^{*}_{N})\|_{L^2(0,T; \cH)} \Bigr), 
\eea
where the constant $\gamma$ is the same as in Lemma~\ref{lem:J_estimates}.
\ec

\bp
By the assumptions, we have
\be
\|u^*  - u^*_N\|_{L^q(0,T; V)}^q \le \frac{1}{\sigma} \left(J(x, u^*_N) - J(x, u^*)\right).
\ee 
Note also that
\bea
J(x, u^*_N) - J(x, u^*) &= J(x, u^*_N) - J_N(x_{N},u^*_N) + J_N(x_{N}, u^*_N) - J(x, u^*) \\
& =  J_{0,x}(u^*_N) - J^N_{0,x_{N}}(u^*_N)+ v_N(0,x_{N}) - v(0,x),
\eea
where we used the fact that 
\bes
 J(x, u^*_N) = J_{0,x}(u^*_N),  \quad J_N(x_{N},u^*_N) = J^N_{0,x_{N}}(u^*_N),  
 \ees 
 and 
\bes
J_N(x_{N},u^*_N)  = v_N(0,x_{N}),   \quad J(x, u^*) = v(0,x).
\ees
The result follows by applying the estimate \eqref{Eq_J_est} to $J_{0,x}(u^*_N) - J^N_{0,x_{N}}(u^*_N)$ and the estimate \eqref{PM_value_est_goal} to $v_N(0,x_{N}) - v(0,x)$.
\ep
\br
Note that \eqref{Eq_growth_onJ} ensures uniqueness of the local minimizer $u^\ast$ in $\mathcal{W}$. 
\er
%%%%%%%%%%%%%%%%%%%%%%%%%%%%%%%%%%%%%%%%%%
\subsection{Examples that satisfy Assumption {\bf (A7)}} \label{Sect_examples}
We consider in this subsection a special but important case of Galerkin approximations \eqref{ODE_Galerkin} built from the 
eigenfunctions\footnote{i.e.~with $\mathcal{H}_N:=\mbox{span}\{e_1, \cdots,e_{N}\}$ and $L_N:= \Pi_N L \Pi_N$, for which Assumptions {\bf (A0)}-{\bf (A2)} are satisfied.} $\{e_k\}_{k\geq 1}$ of $L$, and for which we assume the following properties:
\bi

\item[(i)]  The set of admissible controls $\mathcal{U}_{ad}$ is given as in \eqref{Eq_U_bounded} with $q>1$. 

\item[(ii)]  The linear operator $L: D(L) \subset \cH \rightarrow \cH$  is self-adjoint with compact resolvent and satisfies Assumption {\bf (A0)}.

\item[(iii)] The mapping $F: \cH \rightarrow \cH$ and  $\mathfrak{C}: V \rightarrow \cH$ are locally Lipschitz, and $\mathfrak{C}(0) = 0$.

\item[(iv)]  Assumption {\bf (A6)} is satisfied. 

\ei

We have then the following useful Lemma for applications; see Sect.~\ref{Sec_appl}. 
\bl\label{Lem_examples}
The convergence property \eqref{Est_uniformity_highmodes} in  {\bf (A7)} holds under assumptions (i)-(iv) above.
\el

\bp
Since $L$ is assumed to be self-adjoint with compact resolvent, it follows from spectral theory of self-adjoint compact operator \cite[Thm.~6.8, Prop.~6.9, and Thm.~6.11]{Brezis10} that the eigenfunctions of $L$ form an orthonormal basis of $\cH$, and the eigenvalues $\{\beta_k\}_{k\geq 1}$ of $L$ approach either $\infty$ or $-\infty$ as $k$ approaches $\infty$. Since $L$ is also assumed to be the infinitesimal generator of a $C_0$-semigroup (Assumption {\bf(A0)}), $\beta_k$ is bounded above \cite[Thm.~5.3]{Pazy83}. It follows then that
\be
\beta_N \rightarrow -\infty \quad \text{ as } \quad  N \rightarrow \infty.
\ee

As before $\Pi_N: \cH \rightarrow \cH_N$ denotes the orthogonal projector associated with $\cH_N$ spanned by the first $N$ eigenfunctions. Let us recall also that $\Pi_N^\perp$ denotes the orthogonal projector associated with the orthogonal complement of $\cH_N$ in $\cH$, namely
\bes
\Pi_N^\perp := \mathrm{Id} - \Pi_N. 
\ees
Now, by applying $\Pi_N^\perp$ to both sides of \eqref{Eq_mild} and introducing the notation, 
\bes
y_{N>} := \Pi_N^\perp y, 
\ees
we get, for all $t$ in   $[0, T]$,
\be
y_{N>}(t) = T(t) y_{N>}(0) + \int_{0}^t T(t-s) \Pi_N^\perp \big( F(y(s;x,u)) +   \mathfrak{C} (u(s)) \big) \d s,
\ee
where we also used the fact that $\Pi_N^\perp$ commutes with the semigroup $\{T(t)\}_{t \ge 0}$. 

Let $\mathcal{C}$ be the constant arising in the upper bounds of \eqref{Eq_y_uniform-in-u_bounds} in Assumption {\bf (A6)}, and denote by $\mathfrak{B}_{\cH}$ the closed ball in $\cH$ centered at the origin with radius $\mathcal{C}$. Let also $\mathfrak{B}_{\scriptscriptstyle V}$ be the smallest closed ball in $V$ containing the compact set $U$ given in \eqref{Eq_U_bounded}. We get then for $t\in [0, T]$:
\bea \label{wN_bound}
|y_{N>}(t)| & \le e^{\beta_{N+1} t} |y_{N>}(0)| + \Lip(F\vert_{\mathfrak{B}_{\cH}}) (\mathcal{C} + \|F(0)\|_{\cH}) \int_{0}^t e^{\beta_{N+1} (t-s)}  \d s
\\
&\hspace{10em}+ \Lip(\mathfrak{C}\vert_{\mathfrak{B}_V}) \int_{0}^t e^{\beta_{N+1} (t-s)} \|u(s)\|_V  \d s \\
& \le e^{\beta_{N+1} t} |y_{N>}(0)| +  \frac{1}{|\beta_{N+1}|} \Lip(F\vert_{\mathfrak{B}_{\cH}}) (\mathcal{C} + \|F(0)\|_{\cH}) \\
& \hspace{10em}+ \frac{1}{(|\beta_{N+1} |q')^{\frac{1}{q'}}} \Lip(\mathfrak{C}\vert_{\mathfrak{B}_V})  \|u\|_{L^{q}([0,T; V])},
\eea
where we have used the H\"older inequality with $q' = q/(q-1)$ to estimate the term $\int_{0}^t e^{\beta_{N+1} (t-s)} \|u(s)\|_V  \d s$.

Now, let $C_U:= \sup_{w \in U} \|w\|_V$.  We have 
\be
\|u\|_{L^{q}(0,T; V)} \le C_U T^{\frac{1}{q}}, \qquad \forall u \in \mathcal{U}_{ad}.
\ee
Using the above bound in \eqref{wN_bound}, we obtain %for all $u$ in $\mathcal{U}_{ad}$ that
\bea \label{wN_bound2}
|y_{N>}(t)| \le e^{\beta_{N+1} t} |y_{N>}(0)| & + \frac{1}{|\beta_{N+1}|} \Lip(F\vert_{\mathfrak{B}_{\cH}}) (\mathcal{C} + \|F(0)\|_{\cH}) \\
& \hspace{-8ex}+ \frac{1}{(|\beta_{N+1} |q')^{\frac{1}{q'}}} C_U T^{\frac{1}{q}} \Lip(\mathfrak{C}\vert_{\mathfrak{B}_V}), \;\; \forall t \in [0, T], \; u \in \mathcal{U}_{ad}. 
\eea

Since $\beta_N$ approaches $-\infty$ as $N$ approaches $\infty$ by assumption, we get that for each $\epsilon >0$, $T>0$ and $x$ in $\cH$, there exists a positive integer $N_0$  such that
\be \label{Eq_unif_est_ys_ex}
\sup_{u\in \mathcal{U}_{ad}} \; \sup_{0\leq t\leq T} \|(\mbox{Id}_{\cH} - \Pi_{N}) y(t;x,u)\|_{\cH} \le \epsilon, \quad \Forall N \ge N_0.
\ee
Thus Assumption {\bf (A7)} is satisfied. 
\ep

\br\label{Rmk_other_spectral_assumption}
From the proof given above, it is clear that Assumption (ii) in Lemma~\ref{Lem_examples} can be relaxed to e.g.~
\bi
\item[(ii$'$)] The eigenfunctions of the linear operator $L: D(L) \subset \cH \rightarrow \cH$ forms an orthonormal basis of $\cH$ and the eigenvalues of $L$ approaches $-\infty$.
\ei

Note also that Lemma~\ref{Lem_examples} still holds if $L$ has complex eigenvalues. We just need to work with a complexification of the operator $L$ and the underlying state space $\cH$ (cf.~e.g.~\cite[footnote 17 on p.~55]{CLW15_vol2}), and make the corresponding changes in Assumption (ii$'$) above.
\er
%%%%%%%%%%%%%%%%%%%%%%%%%%%%%%%%%%%%%%%%%%

\section{Application to optimal control of energy balance climate models}\label{Sec_appl}
We show in this section that our framework allows us to provide rigorous Galerkin approximations to the optimal control of  broad class of semilinear heat problems, posed on a compact (smooth) manifold without boundary.  As an application, we show in Sect.~\ref{Sec_EBM} that the optimal control of energy balance models (EBMs) arising in the context of geoengineering  and climate change, can be thus approximated by optimal control problems of ODEs, more tractable numerically.  We first recall some fundamentals of differential geometry to prepare the analysis.

\subsection{Preliminary from differential geometry} \label{Sect_diff_geo_prelim}
To properly write an EBM on the sphere, we recall how differential operators are defined on an abstract compact smooth manifold $\mathfrak{M}$ without boundary, of dimension $n$ and endowed with a Remannian metric\footnote{Recall that a Remannian metric $\mathfrak{g}$ is a smooth family of inner products on the tangent spaces $T_p \mathfrak{M}$. Namely, $\mathfrak{g}$ associates (smoothly) to each $p$ in  $\mathfrak{M}$ a positive definite symmetric
bilinear form $\varphi_p$ on $T_p \mathfrak{M}\times T_p \mathfrak{M}$. In local coordinates, 
\bes
\mathfrak{g}_{ij}(p)=\varphi_p(\partial_i,\partial_j),
\ees
and thus $\mbox{det} \;  \mathfrak{g}(p) >0$. 
} $\mathfrak{g}$.
 First, given a smooth function $u$ on $\mathfrak{M}$,  the gradient $\nabla_{\mathfrak{g}} u$ is a vector field on $\mathfrak{M}$,  that takes in local coordinates the form
\be \label{Eq_grad}
\nabla_{\mathfrak{g}} u:= \sum_{i=1}^n\bigg(\sum_{j=1}^n \mathfrak{g}^{ij}\partial_j u \bigg)\; \partial_i.
\ee
The divergence of a vector field $X=\sum_{j=1}^n X^j \partial_j$ takes the following form in local coordinates
\be \label{Eq_div}
\mbox{div}_{\mathfrak{g}}  \; X= \frac{1}{\sqrt{\mbox{det} \; \mathfrak{g}}} \sum_{i=1}^n \partial_i \bigg( X^i \sqrt{\mbox{det}  \; \mathfrak{g}}\bigg).
\ee
The Laplacian on $(\mathfrak{M},\mathfrak{g})$ takes then the form

\bea
\Delta_{\mathfrak{g}}:&= \mbox{div}_{\mathfrak{g}} \circ \nabla_{\mathfrak{g}}\\
&= \frac{1}{\sqrt{\mbox{det} \; \mathfrak{g}}} \sum_{i=1}^n \partial_i\bigg(\sum_{j=1}^n\mathfrak{g}^{ij}\sqrt{\mbox{det} \; \mathfrak{g}} \; \partial_j\bigg).
\eea

Set 
\be
\mathbb{S}^n:=\{x\in \mathbb{R}^{n+1}\;:\;|x|=1\}.
\ee

Let $y = (y_1,\cdots,y_{n+1})$ be any point of $\mathbb{S}^n$ and $(x_1,\cdots,x_n)$ be its image under the
stereographic projection from the ``north pole'' $N = (0, . . . , 0, 1)$ onto the space 
\be
\mathbb{R}^n \equiv \{(\xi_1,\cdots,\xi_{n+1}) \in \mathbb{R}^{n+1} \; \vert  \; \xi_{n+1}=0\}.
\ee

The canonical Riemannian metric on $\mathbb{S}^n$ takes then the form
\be \label{Def_gij}
\mathfrak{g}_{ij}=\frac{4}{(1+|x|^2)^2} \delta_{ij}, \; \; 1\leq i,j \leq n.
\ee
Note that $\mathfrak{g}^{ij}$ introduced above has here the explicit expression: 
\be
\mathfrak{g}^{ij}=(\mathfrak{g}_{ii})^{-1}\delta_{ij}, \; \; 1\leq i,j \leq n.
\ee
In what follows we denote by $\langle \cdot , \cdot \rangle_{\mathfrak{g}}$ the inner Riemannian product on $(\mathbb{S}^n,\mathfrak{g})$ and by $L^2(\mathbb{S}^n)$ the space of square-integrable real-valued functions for the norm induced by this inner product. The space $\vec{L}^2(\mathbb{S}^n)$ is then defined as 
\bes
\vec{L}^2(\mathbb{S}^n):=\underbrace{L^2(\mathbb{S}^n)\times \cdots \times L^2(\mathbb{S}^n)}_n.
\ees

Introduce now the polar coordinates on $\mathbb{S}^n$. Consider the ``south pole'' $S = (0, . . . , 0,-1)$. For any $y\in \mathbb{S}^n \backslash \{N, S\}$, define $\rho \in (0,\pi)$ and $\theta \in \mathbb{S}^{n-1}$ by
\be
\cos \rho =y^{n+1}, \;\;\; \theta=\frac{y'}{|y'|},
\ee 
where $y'=(y_1,y_2,\cdots,y_n,0)$.

These variables have the following interpretation: the polar radius $\rho$ represents the angle between the position vectors of $y$ and $N$; it can be also regarded as the latitude of the point $y$ measured from $N$. The polar angle $\theta$ can be regarded as the longitude of the point $y$.

The canonical spherical Riemannian metric has the following expression in the polar coordinates:
\be
\d s^2 =\sum_{i,j=1}^n \mathfrak{g}_{ij}\; \d x_i \d x_j =\d \rho^2  + (\sin \rho)^2 \d \theta^2.
\ee

The operator $\Delta_{\mathfrak{g}}$ then takes the form
\be
\Delta_{\mathfrak{g}}=\frac{\partial}{\partial \rho^2} +(n-1) \cot \rho  \frac{\partial}{\partial \rho} +\frac{1}{(\sin \rho)^2 } \Delta_{\theta},
\ee
where $ \Delta_{\theta}$ denotes the Laplace-Beltrami operator on $\mathbb{S}^{n-1}.$

Let us introduce the linear operator 
\bes
\mathcal{L}_\mathfrak{g}:= -\Delta_\mathfrak{g},
\ees
 with domain 
\be\label{dom_L}
\mathcal{D}(\mathcal{L}_\mathfrak{g}):=\{u \in L^2(\mathbb{S}^n) \; : \; \nabla_\mathfrak{g} u \in  \vec{L}^2(\mathbb{S}^n), \; \Delta_\mathfrak{g} u \in L^2(\mathbb{S}^n) \}. 
\ee

It is known that $\mathcal{L}_\mathfrak{g}$ is self-adjoint \cite[Sect.~4.2]{Grigoryan09}. Moreover, the following results hold:

\bt \label{Thm_deltag}
The spectrum of $\mathcal{L}_\mathfrak{g}$ with domain  $\mathcal{D}(\mathcal{L}_\mathfrak{g})$ defined in \eqref{dom_L}  is discrete and consists of an increasing sequence
$\{\lambda_{k}\}_{k=1}^\infty$ of non-negative eigenvalues (counted according to
multiplicity) such that  
\bes
\lim_{k\rightarrow \infty} \lambda_{k}  = +\infty. 
\ees
There is an orthonormal
basis $\{e_{k}\}_{k=1}^\infty$ in $L^2(\mathbb{S}^n)$ such that each function $e_{k}$ is an eigenfunction
of $\mathcal{L}_\mathfrak{g}$ with the eigenvalue $\lambda_k$.
\et

Let $\{E_\lambda \}_{\lambda \ge 0}$ denote the spectral resolution of $\mathcal{L}_\mathfrak{g}$. We can define then for each $t \geq 0$,
\be\label{Def_Pt}
P(t) :=\int_{0}^\infty e^{-t \lambda} \d E_\lambda, \; \; 
\ee
that constitutes a bounded linear operator acting on $L^2(\mathbb{S}^n)$ satisfying the  properties summarized below.

\bt \label{Thm_heat_semigroup}
Let $P(t)$ be the operator defined by \eqref{Def_Pt}, then 
\bi
\item[(i)] For any $t \ge 0$, $P(t)$ is a bounded self-adjoint operator on $L^2(\mathbb{S}^n)$, and 
\be
\|P(t)\|\le 1. 
\ee
\item[(ii)] The family $\{P(t)\}_{t\ge 0}$ satisfies the semigroup identity: 
\be
P(t) P(s) = P(t+s),
\ee
for all $t,s \ge 0$. 

\item[(iii)] The mapping $t \mapsto P(t)$ is strongly continuous on $[0, \infty)$. That is, for any $t\ge 0$ and $f \in L^2(\mathbb{S}^n)$, 
\be
\lim_{s \rightarrow t} P(s) f = P(t) f, 
\ee
where the limit is understood in the norm of $L^2(\mathbb{S}^n)$. In particular, for any $f \in L^2(\mathbb{S}^n)$,
\be
\lim_{t \rightarrow 0^+} P(t) f = f. 
\ee

\item[(iv)] For all $f \in L^2(\mathbb{S}^n)$ and $t > 0$, we have that $P(t) f$  lies in $\mathcal{D}(\mathcal{L}_\mathfrak{g})$ and 
\be
\frac{\d}{\d t} (P(t) f)  = - \mathcal{L}_\mathfrak{g}(P(t) f). 
\ee
\ei

\et

The above theorems are particular cases of results presented in \cite{Grigoryan09} for the Laplace operator defined on general {\it weighted smooth manifolds} \cite[Def.~3.17]{Grigoryan09}. See \cite[Thm.~10.13]{Grigoryan09} for Theorem~\ref{Thm_deltag} and \cite[Thm.~4.9]{Grigoryan09} for Theorem \ref{Thm_heat_semigroup}. 

\subsection{Galerkin approximations of controlled semilinear heat equations on $\mathbb{S}^n$}\label{Sec_optctr_man}
Given $\mathcal{L}_\mathfrak{g}=-\Delta_\mathfrak{g} $ with domain $\mathcal{D}(\mathcal{L}_\mathfrak{g})$ given in \eqref{dom_L}, we consider the following abstract controlled semilinear heat problem of the form  \eqref{ODE} posed in $L^2(\mathbb{S}^n)$:
\bea \label{heat_eq}
\frac{\d y}{\d t} &= - \mathcal{L}_\mathfrak{g} \, y + F(y) + \mathfrak{C} (u(t)),  \quad t \in (0, T], \\
y(0) &= x \; \in L^2(\mathbb{S}^n).
\eea

In what follows we denote by $\cH$ the space $L^2(\mathbb{S}^n)$. Based on Theorem~\ref{Lem:uniform_in_u_conv} of Section~\ref{Sect_Galerkin}, we show the uniform convergence of Galerkin approximations  to \eqref{heat_eq} associated with the reduced state space $\cH_{N} \subset L^2(\mathbb{S}^n)$ defined by 
\be \label{Eq_Hn_heat}
\cH_{N}  := \mathrm{span}\{e_k \;:\; k = 1, \cdots, N\}, \; N\in \mathbb{Z}_+^\ast,
\ee
in which the $e_k$'s denote the eigenfunctions of $-\mathcal{L}_\mathfrak{g}$ lying in $\mathcal{D}(\mathcal{L}_{\mathfrak{g}})$; see Theorem~\ref{Thm_deltag}.

The linear approximations $L_N$ of the operator $-\mathcal{L}_\mathfrak{g}$ are then naturally defined as 
\be
L_N:= \Pi_N \Delta_\mathfrak{g} \Pi_N: \cH \rightarrow \cH_N,
\ee
where $\Pi_N$ denotes the orthogonal projector associated with $\cH_N$. % as before.  

Theorem~\ref{Lem:uniform_in_u_conv} leads then to the following corollary about uniform convergence of Galerkin approximations of \eqref{heat_eq}.

%%%%%%%%%%

\bc  \label{Cor:heat_Galerkin_approx}

Let $V$ be a Hilbert space. Assume that $F: L^2(\mathbb{S}^n) \rightarrow L^2(\mathbb{S}^n)$ 
and  $\mathfrak{C}: V \rightarrow L^2(\mathbb{S}^n)$ are both globally Lipschitz, and $\mathfrak{C}(0) = 0$.  Assume the set of admissible controls $\mathcal{U}_{ad}$ is given by \eqref{Eq_U_bounded} with  
$U$ therein being a compact subset of the Hilbert space~$V$ and with $q > 1$.  

Then, for any $T > 0$ and any $(x,u)$ in $L^2(\mathbb{S}^n) \times \mathcal{U}_{ad}$, the problem \eqref{heat_eq} admits a unique mild solution $y(\cdot; x, u)$ in $C([0,T],\cH)$, and its Galerkin approximation \eqref{ODE_Galerkin} associated with the eigen-subspaces \eqref{Eq_Hn_heat} admits a unique solution $y_N(\cdot; \Pi_N x, u)$ in  $C([0,T],\cH_N)$ for each $N$ in $\mathbb{Z}_+^\ast$. Moreover,  the following uniform convergence result holds:
\be  \label{heat_eqn_uniform_in_u_conv}
\lim_{N\rightarrow \infty}  \sup_{u\in \mathcal{U}_{ad}} \sup_{t \in [0, T]} \|y_N(t; \Pi_N x, u) - y(t; x,u)\|_{\cH} = 0. 
\ee

\ec

\bp

We only need to check the conditions  {\bf (A0)}--{\bf (A2)},  {\bf (A6)}, and {\bf (A7)} assumed in Theorem~\ref{Lem:uniform_in_u_conv}. This is done below in four steps. 

\medskip
{\sc Step 1}: Checking {\bf (A0)}--{\bf (A1)}. From Theorem~\ref{Thm_heat_semigroup}, we know that $-\mathcal{L}_\mathfrak{g}$ is the infinitesimal generator of the $C_0$-semigroup of contractions, $\{P(t)\}_{t\ge 0}$, on $L^2(\mathbb{S}^n)$.

Since $-\mathcal{L}_\mathfrak{g}$ is self-adjoint, the operator $L_N$ is a finite-rank diagonal operator acting on $\cH$ and 
we have for each $\phi$ in $\cH$,
\be
e^{t L_N} \phi := \sum_{k=1}^N e^{-t \lambda_k} \langle \phi,e_k\rangle_{\mathfrak{g}} \; e_k. 
\ee

 It follows that each operator $L_N$ generates a $C_0$-semigroup of contractions on $\cH$, which will be denoted by $\{P_N(t)\}_{t\ge 0}$. 
We have thus checked Assumptions {\bf (A0)}--{\bf (A1)} given in Section~\ref{Sect_Galerkin} with $M=1$ and $\omega =0$, namely
\be \label{Eq_control_linearflow_heat}
\quad \|P(t)\| \le 1 \quad \text{and} \quad \|P_N(t)\| \le 1, \quad N \geq 0, \; \quad t \ge 0.
\ee

\medskip
{\sc Step 2}: Checking {\bf (A2)}. This condition results from the self-adjointness of $-\mathcal{L}_\mathfrak{g}$. Indeed, for any given $\phi$  in $\mathcal{D}(\mathcal{L}_\mathfrak{g})$, since both $\phi$ and $\Delta_\mathfrak{g}\phi$ belong to $L^2(\mathbb{S}^n)$, the following expansions against the eigenbasis hold: 
\be
\phi = \sum_{i=1}^\infty a_i e_i, \quad \Delta_\mathfrak{g} \phi  = \sum_{i=1}^\infty b_i e_i,
\ee
where 
\be
a_i = \langle \phi, e_i \rangle_{\mathfrak{g}} ,  \qquad b_i = \langle \Delta_\mathfrak{g} \phi , e_i \rangle_{\mathfrak{g}}, \qquad i \in \mathbb{Z}_+^\ast.
\ee

Note that 
\be
b_i = \langle \Delta_\mathfrak{g} \phi, e_i \rangle_{\mathfrak{g}}  =  \langle \phi, \Delta_\mathfrak{g} e_i \rangle_{\mathfrak{g}}  =  \lambda_i \langle \phi, e_i \rangle_{\mathfrak{g}}  = \lambda_i a_i. 
\ee
We get then that
\bea
\|L_N \phi - \Delta_\mathfrak{g}\phi \|_{L^2(\mathbb{S}^n)}  
&= \Big\| \Pi_N \Delta_\mathfrak{g} \Big(\sum_{i=1}^N a_i e_i \Big) - \Delta_\mathfrak{g} \phi \Big \|_{L^2(\mathbb{S}^n)} \\
&= \Big\|  \sum_{i=1}^N \lambda_i a_i e_i - \sum_{i=1}^\infty b_i e_i \Big \|_{L^2(\mathbb{S}^n)} 
=  \Big \|\sum_{i=N+1}^\infty b_i e_i \Big \|_{L^2(\mathbb{S}^n)}
\eea
and Assumption ${\bf (A2)}$ follows. 

\medskip
{\sc Step 3}: Checking {\bf (A6)}. Since $F$ and $\mathfrak{C}$ are globally Lipschitz, the existence of a unique mild solution to \eqref{heat_eq} and to its Galerkin approximation \eqref{ODE_Galerkin} follows directly from a classical fixed point argument and standard Gronwall's estimates; see e.g.~\cite[Prop.~4.3.3 and Thm.~4.3.4]{Cazenave_al98}. In particular, by \eqref{Eq_control_linearflow_heat}, we obtain the following a priori estimates  for all $t$ in $[0, T]$ and $N$ in $\mathbb{Z}_+^\ast$:
\bea \label{solution_bounds_heat_eq}
& \|y(t; x, u)\|_{L^2(\mathbb{S}^n)} \le e^{\Lip(F)t}\|x\|_{L^2(\mathbb{S}^n)} + \int_0^t g(s) \d s \\
& \hspace{16em}+ \Lip(F) \int_0^t g(s) e^{\Lip(F)(t-s)} \d s, \\
& \|y_N(t; \Pi_N x, u)\|_{L^2(\mathbb{S}^n)} \le e^{\Lip(F)t}\|x\|_{L^2(\mathbb{S}^n)} + \int_0^t g(s) \d s\\
& \hspace{16em} +  \Lip(F) \int_0^t g(s) e^{\Lip(F)(t-s)} \d s,  
\eea
where 
\bes
g(s) =  \|F(0)\|_{L^2(\mathbb{S}^n)} + \Lip(\mathfrak{C}) \|u(s)\|_{L^2(\mathbb{S}^n)}, \qquad \text{for a.e. } s \in [0, T]. 
\ees
Moreover, since by assumption $u(t)$ takes value in a compact thus bounded set $U$ for each $u \in \mathcal{U}_{ad}$, the a priori estimates \eqref{solution_bounds_heat_eq} also ensure the required uniform boundedness estimates \eqref{Eq_y_uniform-in-u_bounds} stated in Assumption ${\bf (A6)}$.

\medskip
{\sc Step 4}: Checking {\bf (A7)}. Due to our assumptions and from what precedes, the conditions of Lemma~\ref{Lem_examples} are satisfied and thus 
Assumption ${\bf (A7)}$ is satisfied. The proof is complete. 

\ep

\br\label{local_lipman}
Corollary \ref{Cor:heat_Galerkin_approx} has been formulated in the case where $F$ and $\mathfrak{C}$ are globally Lipschitz, but actually the conclusions of this corollary still hold if these conditions are relaxed to be locally Lipschitz  as long as Assumption {\bf (A6)} is satisfied with the relevant a priori estimates as a consequence of Theorem \ref{Lem:uniform_in_u_conv}.
\er

\br \label{Rmk_heat_F}

Similar to Remark~\ref{Rmk:time-dependent-F}, we note that Corollary~\ref{Cor:heat_Galerkin_approx} still holds when the nonlinearity $F$ depends also on time, i.e., $F: [0,T] \times L^2(\mathbb{S}^n) \rightarrow L^2(\mathbb{S}^n)$, and satisfies for instance that $F(\cdot,y) 
\in L^\infty(0,T; L^2(\mathbb{S}^n))$ for each $y\in L^2(\mathbb{S}^n)$, $F(t,\cdot)$ is globally Lipschitz for almost every $t \in [0, T]$ and the mapping $t\mapsto \Lip(F(t, \cdot))$ is in $L^\infty(0,T)$.

\er

\br  \label{Rmk_heat_general_metric}

Note also that Corollary~\ref{Cor:heat_Galerkin_approx} still holds when the semilinear heat problem \eqref{heat_eq} is posed on a general $n$-dimensional Riemannian smooth and compact manifold $(\mathfrak{M}, \mathbf{g})$ without boundary, with the Riemannian metric $\mathbf{g}$ not limited thus to $\mathfrak{g}$ defined in \eqref{Def_gij}. Similarly the case of semilinear heat problem  posed on a non-empty relatively compact subset $\Omega$ of $(\mathfrak{M}, \mathbf{g})$ with homogenous Dirichlet boundary conditions can also be dealt with.  This is because Theorems~\ref{Thm_deltag} and \ref{Thm_heat_semigroup} still hold for such cases; see again \cite[Thm.~4.9 and Thm.~10.13]{Grigoryan09}. In particular, Corollary~\ref{Cor:heat_Galerkin_approx} holds when the Riemannian metric $\mathfrak{g}$ on $\mathbb{S}^n$ is replaced by another Riemannian metric $\widetilde{\mathfrak{g}}$ and the Laplacian $\Delta_\mathfrak{g}$ in \eqref{heat_eq} is replaced by $\Delta_{\widetilde{\mathfrak{g}}}$ accordingly. This is remark about the change of Remannian metric is used in Sect.~\ref{Sec_EBM} that follows.

\er

\subsection{Energy balance models} \label{Sec_EBM}

Energy balance models (EBMs) are among the simplest climate models that can be used for the study of climate sensitivity.  They are formulated based on the energy balance on the Earth surface \cite{North_al81,stephens2012update} and have the Earth surface temperature as the only dependent variable. First made popular by the works \cite{Budyko69,Sellers69}, these models have been extensively studied since both analytically and numerically; see e.g.~\cite{bermejo2009mathematical,North_al81,Roques_al14} and references therein. 

With suitable tuning of their parameters, EBMs that resolve the Earth's land-sea geography and are forced by the seasonal insolation cycle have been shown to mimic, to a certain extent, the observed zonal temperatures for the observed present climate \cite{North_al83,crowley2000causes}. Once EBMs are fitted to observations \cite{Sellers69,Ghil76,graves1993new} or to simulations from general circulation models (GCMs)\cite{hyde1989comparison,crowley2000causes}, they can be used to estimate the temporal response patterns to various forcing scenarios; such a methodology is of particular interest in the detection and attribution of climate change \cite{stone2007detection}.

Depending on whether zonal or meridional averages are used, the modeled surface temperature can either depend on the latitude only, or depend on both  the latitude and the longitude, resulting respectively in  1D, or 2D models. In the 2D case, the model is posed on the two-dimensional unit sphere $\mathbb{S}^2$, and takes typically the following form \cite{North_al83}:
\be \label{Eq_EBM}
\frac{\partial T(\xi,t)}{\partial t} = \mathrm{div}_\mathfrak{g}( D(x) \nabla_{\mathfrak{g}} T(\xi, t) )+ f(t, x, T(\xi,t)) - g(T(\xi,t)) + E(\xi,t), 
\ee
for all $\xi$ in $\mathbb{S}^2$, and $t > 0.$

Here the gradient $\nabla_{\mathfrak{g}}$ and the divergence $\mathrm{div}_\mathfrak{g}$ on the Riemannian manifold $(\mathbb{S}^2, \mathfrak{g})$ are given respectively by \eqref{Eq_grad} and \eqref{Eq_div}, and the Riemannian metric $\mathfrak{g}$ is given by \eqref{Def_gij}.  The diffusion term $\mathrm{div}_\mathfrak{g} ( D(x) \nabla_{\mathfrak{g}} T(\xi, t) )$ describes the redistribution of heat on the surface of the Earth by conduction and convection, the reaction terms $f(t, x, T(\xi,t)) - g(T(\xi,t))$ express the balance between incoming and outgoing radiations, and $E(\xi,t)$ denotes an anthropogenic forcing. See Table~\ref{tab_EBM} for the precise meaning of the symbols involved in  \eqref{Eq_EBM}.  We 
refer to \cite{bermejo2009mathematical} for the rigorous approximation of \eqref{Eq_EBM} via finite elements on manifolds.

\begin{table}[h]
\caption{Glossary of model's parameters \& variables}
\label{tab_EBM}       % Give a unique label
\centering
\begin{tabular}{ll}
%\hline\noalign{\smallskip}
\toprule\noalign{\smallskip}
Symbol & Interpretation \\ %&  Sections  \\
\noalign{\smallskip}\hline\noalign{\smallskip}

$\xi := (\rho, \theta)$ & $\rho \in (0, \pi)$ denotes the latitude and $\theta \in (-\pi, \pi)$, the longitude \\
$T(\xi,t)$ & sea surface temperature at $\xi,t$ \\
$x$ &  sine of the latitude, i.e. $x = \sin (\rho)$ \\
$D(x)$ &  diffusion coefficient for heat transport, zonally averaged \\
& and hence does not dependent on the longitude $\theta$ \\
$f(t, x, T(\xi,t))$ & incoming solar radiation \\
$g(T(\xi,t))$ &  outgoing infrared radiation \\
$E(\xi,t)$ &  forcing representing greenhouse gas emissions   \\
\noalign{\smallskip} \bottomrule 
\end{tabular}
\end{table}

The functions $f(t, x, T(\xi,t))$ and $g(T(\xi,t))$ are typically of the following form \cite{Budyko69}:
\bea \label{Eq_radiation}
& f(t, x, T(\xi,t)) = Q S(x,t) (1 - \alpha(x,T)), \\
& g(T(\xi,t)) = a + b T(\xi,t).
\eea
Here $Q$ is the so-called solar constant, $S(x,t)$ denotes a solar insolation distribution function, $\alpha(x,T)$ denotes the albedo, and $a$ and $b$ are empirical constants typically estimated from satellite observations; see e.g.~\cite{graves1993new}. We refer to \cite{alexeev2005polar} for the calibration of other parameters including $Q$, or coefficients such as $S(x,t)$ or contained in $\alpha(x,T)$; see also \cite{Roques_al14}. 

We also note that the LHS of \eqref{Eq_EBM} should be multiplied by a factor $\kappa(\xi)$, which measures the effective heat capacity per unit area. Here, we have assumed that $\kappa(\xi)$ is a constant which is taken to be $1$ after a scaling in the time variable.

As a preparation to cast a controlled version of \eqref{Eq_EBM} into the form of \eqref{heat_eq}, we will make use of a new Riemannian metric so that the diffusion term in \eqref{Eq_EBM} becomes simply the Laplician under this new metric. For this purpose, we assume that
\bi

\item[{\bf (H1)}] the diffusion coefficient $D(x)$ is $C^1$-smooth and is strictly positive. 

 \ei
 
 By introducing the new Riemannian metric 
\be
\widetilde{\mathfrak{g}} = \frac{1}{D(x)} \mathfrak{g}, 
\ee
it is known that the Laplace operator under this new Riemannian metric is given by 
\be
\Delta_{\widetilde{\mathfrak{g}}} = \mathrm{div}_\mathfrak{g}  ( D(x) \nabla_{\mathfrak{g}}).
\ee
Therefore, Eq.~\eqref{Eq_EBM} can be rewritten into the following form on $(\mathbb{S}^2, \widetilde{\mathfrak{g}})$:
\be \label{Eq_EBM_v2a}
\frac{\partial T(\xi,t)}{\partial t} = \Delta_{\widetilde{\mathfrak{g}}} T(\xi,t) + f(t, x, T(\xi,t)) - g(T(\xi,t)) + E(\xi,t).
\ee

With this rewriting, we consider the operator $\mathcal{L}_{\widetilde{\mathfrak{g}}}=-\Delta_{\widetilde{\mathfrak{g}}}$ with domain $\mathcal{D}(\mathcal{L}_{\widetilde{\mathfrak{g}}})$ given by \eqref{dom_L} in which $\widetilde{\mathfrak{g}}$ replaces $\mathfrak{g}$. 
We are now in position to apply the general results of Sect.~\ref{Sec_optctr_man} in particular Corollary \ref{Cor:heat_Galerkin_approx} to a contemporary problem related to geoengineering that we address here in the framework of optimal control of EBMs such as \eqref{Eq_EBM_v2a}. 

\subsection{Optimal control of climate?}
In 1955, John von Neumann envisioned  that ``probably intervention in atmospheric and climate matters will come in a few decades, and will unfold on a scale difficult to imagine at present;'' see \cite{Neumann55}. As our planet enters a period of changing climate never before experienced in recorded human history, primarily caused by the rapid buildup of carbon dioxide in the atmosphere from the burning of fossil fuels, interest is growing in the potential for deliberate large-scale intervention in the climate system to counter climate change; see e.g.~\cite{Shepherd09, NAP15a,NAP15b}. Although we are still far away from large-scale implementation of what John von Neumann envisioned decades ago, the consideration of climate engineering---also known as geoengineering---is raising in the scientific community with
a literature that became more abundant on the topic over the recent years; see e.g.~\cite{Lenton09,Shepherd09, Vaughan_al11,NAP15a,NAP15b}.

At the simplest level, the surface temperature of the Earth results from the net balance of incoming solar (shortwave) radiation
and outgoing terrestrial (longwave) radiation \cite{Kiehl97}. Proposed geoengineering methods attempt to rectify the current and potential future radiative imbalance and they are usually divided into two basic categories: (i) carbon dioxide removal techniques which remove CO$_2$ from the atmosphere to increase the  amount of longwave radiation emitted by the Earth; and (ii) solar radiation management techniques that reduce the amount of solar (shortwave) radiation absorbed by the Earth by reflecting a small percentage of the sun's light and heat back into space. 

While a lot of efforts have been devoted to describing different geoengineering options in detail and discussing their advantages, effectiveness, potential side effects and drawbacks, still more understanding is required before any method could even be seriously considered for deployment on the requisite international scale \cite{Shepherd09}.  On the other hand, policies to reduce global greenhouse gas (GHG) emissions is a pressing topic on any political agenda, and uncertainties to climate change \cite{murphy2004quantification,meinshausen2009greenhouse} add to the difficulty in quantifying unambiguously the effects of forcing variations on the climate system.

From a mathematical perspective, since any geoengineering methods or GHG mitigation policies can be expressed as controls acting on the climate system, it is natural and important to investigate whether a given type of controls, corresponding e.g.~to one or a combination of several geoengineering methods, can drive the climate system from a given ``current'' state to a desired state over a targeted finite time horizon. This controllability aspect has indeed been investigated within the context of climatology based on some types of EBMs; see e.g.~\cite{Diaz94b, Diaz94}. 

Given that any large-scale decision for addressing climate change have economic \cite{hallegatte2009strategies}, societal or physical constraints, it also seems natural to frame the problem as an optimal control of the climate system to seek for controls within a chosen set of geoengineering strategies that lead to the minimization of a relevant cost functional.  To our knowledge, this optimal control perspective has not yet been investigated much from a fundamental viewpoint. In the following, we aim to provide a sufficiently general formulation for this purpose, based on the class of EBMs encompassed by Eq.~\eqref{Eq_EBM}. The latter equation will serve as our underlying state equation in what follows. 

Since EBMs are known to provide reliable models of the mean annual global temperature distribution around the globe \cite{hyde1989comparison,crowley2000causes,stone2007detection}, 
they constitute a natural laboratory for such an investigation before one moves onto more sophisticated and detailed climate models such as GCMs \cite{GCS08}. In that respect, it is also worth mentioning that EBMs can actually be derived from the thermodynamics equation of the atmosphere primitive equations via an averaging procedure \cite{Kiehl92}.  See also \cite{Brock13} for the design of optimal economic mitigation policies based on EBMs coupled with an economic growth model.

\subsection{Optimal control of EBMs: Convergence results of value functions}\label{Sec_OC_EBM}
We formulate the optimal control of EBMs within the general setting of Sect.~\ref{Sect_Galerkin}, by relying on the properties of the heat semigroup on the sphere recalled in Sect.~\ref{Sec_optctr_man}, here generated by $-\mathcal{L}_{\widetilde{\mathfrak{g}}}$. 

First, in order to allow for geoengineering strategies of different nature in different geographic regions, we consider a collection of open subsets 
\bes
\{\Omega_i \subset \mathbb{S}^2 \; : \; i = 1,\cdots, M\},
\ees
 where $M$ denotes the number of such strategies, one for each region $\Omega_i$, with possible overlapping.

The set of admissible controls is defined as follows.  For region $\Omega_i$ we define the Hilbert space of functions
\be
V_i=L^2(\Omega_i, \mathbb{R}^{n_i}), 
\ee
with $n_i$ some positive integer, and introduce 
\be
V := V_1 \times  \cdots \times V_M.
\ee
Consider for each $i$, $U_i$ to be a compact subset of $V_i$ and let us introduce the set  
\be
U := \{v = (v_1, \cdots, v_M) \in V \; :\;  v_i \in U_i, \; i = 1, \cdots, M\}.  
\ee

Let the set of admissible controls be given by: 
\be\label{Eq_U_bounded2}
\mathcal{U}_{ad}:=\{u\in L^q(0,t_f; V)\;: u(s) \in U \textrm{ a.e. }\}.
\ee
where $q >  1$ is fixed.  
Note that since each $V_i$ is a space of functions defined over a region $\Omega_i$, an admissible control $u$ in $\mathcal{U}_{ad}$ is actually a locally distributed control.  Note that given $u$ in $\mathcal{U}_{ad}$ and $t$ in $(0,t_f)$, its $i^{{\rm th}}$-component $u_i(t)$ is a spatial function that lies in $V_i$. We will denote hereafter by $u_i(t)[\xi]$ its value taken at $\xi \in \mathbb{S}^2$.

Finally, we assume that the combined effects of the geoengineering strategies on the global temperature field $T(\xi,t)$, is represented via a nonlinear function 
\bea
G: \; &\mathbb{R}^{n_1} \times \cdots \times\mathbb{R}^{n_M} \rightarrow \mathbb{R}\\
    &(\zeta_1,\cdots,\zeta_M)\longmapsto G(\zeta_1,\cdots,\zeta_M),
\eea
that forces Eq.~\eqref{Eq_EBM_v2a}. In practice, the modeler has to specify the function $G$ (and $V$), depending on the geoengineering strategy or the GHG mitigation policy adopted as well as the EBM retained. As explained below, our framework ensures that a global Lipschitz  assumption\footnote{That can be relaxed to a local Lipschitz  assumption as long as a priori error estimates are available to ensure Assumption {\bf (A6)}; see Remark \ref{local_lipman}.} on $G$ allows for convergence of Galerkin approximations and thus provide a rigorous basis for a numerical investigation of various control scenarios.

We consider thus, for each $u$ in  $\mathcal{U}_{ad}$, the following controlled version of the EBM \eqref{Eq_EBM_v2a} which writes for each $\xi \in \mathbb{S}^2$ and $t \in [0, t_f]$ as,
\bea \label{Eq_EBM_v2}
\frac{\partial T(\xi,t)}{\partial t} &= \Delta_{\widetilde{\mathfrak{g}}} T(\xi,t) + f(t, x, T(\xi,t)) - g(T(\xi,t)) + E(\xi,t)\\
&\hspace{16em} + G(\widetilde{u}_1(t,\xi),\cdots, \widetilde{u}_M(t,\xi)), 
\eea
supplemented with an initial condition $T_0$ in $L^2(\mathbb{S}^2).$  Here
\be
\widetilde{u}_i(t,\xi)=
\begin{cases}
u_i(t)[\xi], \; \mbox{ if } \xi  \in \Omega_i,\\
0, \; \; \mbox{ otherwise}.
\end{cases}
\ee

In order to recast the IVP associated with \eqref{Eq_EBM_v2} into the abstract form \eqref{heat_eq}, we introduce the following function spaces
\be
\cH:= L^2(\mathbb{S}^2), \qquad  \cH_1 := \mathcal{D}(\mathcal{L}_{\widetilde{\mathfrak{g}}}),
\ee
where $\mathcal{D}(\mathcal{L}_{\widetilde{\mathfrak{g}}})$ is defined in \eqref{dom_L}.

We make the following assumption. 
\bi

\item[{\bf (H2)}]  The functions $f$ and $g$ take the forms given by \eqref{Eq_radiation}, where the solar insolation distribution function $S$ therein lives in the space $L^\infty((0, \infty), L^\infty(\mathbb{S}^2))$ and the albedo $\alpha(x,T)$ is a continuous, piecewise-linear ramp function such as given in \cite[Eq.~(2a)]{Ghil76} or \cite[Eq.~(2.2)]{Roques_al14}.  The GHG emission term $E$ belongs to $L^1_{\mathrm{loc}}((0, \infty), L^2(\mathbb{S}^2))$. 

\ei
We define now the nonlinearity $F:  [0,t_f] \times \cH  \rightarrow \cH$ to be: 
\be \label{Eq_EBM_F}
F(t, v)[\xi] :=  f(t, x, v(\xi)) - g(v(\xi)) + E(\xi,t), 
\ee
for all   $v$ in $\cH$, and a.e.~$\xi$  in $\mathbb{S}^2$, $t \in [0, t_f]$.

Under Assumption {\bf (H2)}, for each $\v$ in $\cH$ and almost every $t\in [0, t_f]$, $F(t,v)$ belongs to $\cH$. 

Finally, we define the nonlinear operator $\mathfrak{C}: V \rightarrow \cH$ associated with the control to be: 
\be \label{Eq_EBM_C}
\mathfrak{C} (v)[\xi] :=  G(\widetilde{v}_1(\xi),\cdots,\widetilde{v}_M(\xi)),  \qquad  \text{ for all } v \in V, \text{ and a.e.} \; \xi \in \mathbb{S}^2,
\ee
with
\be
\widetilde{v}_i(\xi)=
\begin{cases}
v_i(\xi), \; \mbox{ if } \xi  \in \Omega_i,\\
0, \; \; \mbox{ otherwise}.
\end{cases}
\ee

Then, Eq.~\eqref{Eq_EBM_v2} can be rewritten as Eq.~\eqref{heat_eq} of Sect.~\ref{Sec_optctr_man}, with the (nonlinear) operators $F$ and $\mathfrak{C}$ defined above. Having the purpose in mind of driving the temperature field $T(\xi,t)$ to a state sufficiently close to a specified profile at the final time $t_f$, while keeping the control cost ``low'', we consider the cost functional: 
\be \label{Eq_EBM_J}
J(T_0, u) = \int_{0}^{t_f} \hspace{-.5ex}\left ( \frac{1}{2} \|T(\cdot, t; T_0, u) - T_d\|_{\cH}^2 + \frac{\mu}{2} \|u(\cdot, t)\|_{V}^2 \right) \d t,   \;\; T_0 \in \cH, \; \mu \geq 0.
\ee
Here $T_d$  denotes the targeted temperature field over the globe (that lies in $\cH$) and $T(\cdot, t; T_0, u)$ denotes the mild solution to \eqref{Eq_EBM_v2} that emanates from $T_0$.

The associated optimal control problem reads then: 
\bea  \label{P_EBM} 
\begin{aligned}
\min \, J(T_0,u)  \quad \text{ s.t. }&  \quad (T, u) \in L^2(0,T; \cH) \times  \mathcal{U}_{ad} \text{ solves  Eq.} ~\eqref{Eq_EBM_v2}  \\& \text{ subject to the initial condition}  \; T(\cdot, 0)  = T_0 \in \cH.
\end{aligned}
\eea

\medskip
\noindent{\bf Approximation of the value function and error estimates about the optimal control.}  Note that thanks to Assumption {\bf (H2)}, the nonlinearity $F$ defined in \eqref{Eq_EBM_F} satisfies the conditions required in Remark~\ref{Rmk_heat_F}. If we assume furthermore that the nonlinear operator $\mathfrak{C}$ defined by \eqref{Eq_EBM_C} to be globally Lipschitz as a mapping from $V$ to $\cH$ (and $\mathfrak{C}(0) = 0$) then by Corollary \ref{Cor:heat_Galerkin_approx} and Remarks~\ref{Rmk_heat_F}--\ref{Rmk_heat_general_metric}, the uniform convergence result given by \eqref{heat_eqn_uniform_in_u_conv} in Corollary~\ref{Cor:heat_Galerkin_approx} holds for the IVP associated with \eqref{Eq_EBM_v2}. Namely,
\be\label{Cve_y}
\lim_{N\rightarrow \infty}  \sup_{u\in \mathcal{U}_{ad}} \sup_{t \in [0, T]} \|y_N(t; \Pi_N x, u) - y(t; x,u)\|_{\cH} = 0, 
\ee
where $y_N(\cdot; \Pi_N x, u)$ denotes the solution to the Galerkin approximation of ~\eqref{Eq_EBM_v2} associated with the eigen-subspace spanned by the first $N$ eigenfunctions of the Laplacian $\Delta_{\widetilde{\mathfrak{g}}}$.

Note also that Condition \eqref{C1} on the cost functional $J$ defined in \eqref{Eq_EBM_J} is clearly satisfied. Then, Theorem~\ref{Thm_cve_Galerkin_val} applies if we assume furthermore that there exists for each pair $(t,x)$ a  minimizer $u_{t,x}^*$ (resp.~$u_{t,x}^{N,*}$) in $\mathcal{U}_{ad}[t,T]$ of the minimization problem in \eqref{subEq1} (resp.~in \eqref{subEq2}) associated here with the optimal control problem\eqref{P_EBM}. Thus, for any $x$ in $\cH$, it holds that
\be \label{Cve_v}
\lim_{N \rightarrow \infty} \sup_{t \in [0, T]} |v_N(t,\Pi_N x) - v(t,x)| = 0.
\ee
Observe that since  $\mathfrak{C}$ is globally Lipschitz, Assumption {\bf (H2)} leads to simple a priori error estimates such as \eqref{solution_bounds_heat_eq}, showing thus that Assumption {\bf (A6)} is satisfied; see also Remark \ref{Rmk_heat_F}.  The other conditions in Assumption {\bf (E)} of Sect.~\ref{Sec_Err_estimates} are trivially satisfied here and thus the error estimates of Sect.~\ref{Sec_Err_estimates} hold. In particular Corollary \ref{Lem_controller_est} 
applies with $q=2$ and $\sigma=\mu/2$ which, using the notations of this corollary, leads to
\bea\label{Err_estu}
\|u^\ast - u^\ast_{N}\|_{L^2(0,T; V)}^2 & \le \frac{4\mathcal{C}+4\|T_d\|_{\mathcal{H}}}{\mu} \left[\sqrt{T} + \gamma T  \right] \Bigl( \| \Pi_N^\perp y(\cdot;u^*)\|_{L^2(0,T; \cH)}  \\
&  \hspace{11em} + 2 \|  \Pi_N^\perp y (\cdot; u^{*}_{N})\|_{L^2(0,T; \cH)} \Bigr),
\eea 
after a simple estimate  of $\Lip(\mathcal{G}\vert_{\mathfrak{B}})$ where $\mathcal{G}(y)=\|y-T_d\|_{\mathcal{H}}^2$, $y \in \cH$.

With the rigorous convergence results \eqref{Cve_y} and \eqref{Cve_v}  and error estimate  \eqref{Err_estu}, the numerical approximation of solutions to \eqref{P_EBM} becomes affordable via e.g.~a Pontryagin-Maximum-Principle approach applied to Galerkin approximations \cite{CL15} of Eq.~\eqref{Eq_EBM} built here from spherical harmonics.
Indeed, a relatively realistic EBM such as given by Eq.~\eqref{Eq_EBM} is known to be simulated accurately out of few spherical harmonics (typically $20 \leq N \leq 30$); see e.g.~\cite{North_al83,hyde1989comparison}.

\section{Concluding remarks} \label{Sect_concluding_remarks}
Thus, by means of rigorous Galerkin approximations of nonlinear evolution equations in Hilbert spaces, this article provides a natural framework for the synthesis of approximate optimal controls, along with approximations of the value functions.
 The framework opens up several possible directions for future research. We outline some of these issues below. 
 
1. The usage of spectral methods for solving more complex climate models than EBMs considered in Sect.~\ref{Sec_appl} is standard. 
By its natural assumptions to verify in practice, the framework presented above makes possible to address  the problem of geoengineering strategies or GHG mitigation policies in terms of optimal control of Galerkin approximations of such models, enabling thus, at least theoretically, to reduce the dimension of the problem. 
However, very often these models include e.g.~nonlinear advection terms that require to deal with a loss of regularity. Our framework needs thus to be amended to deal with such a situation.  The use of interpolated spaces to deal with the loss of regularity and formulations of the Trotter-Kato theorem exploiting Gelfand triple are natural tools to cope with this difficulty. Analogues of Theorems \ref{Lem:uniform_in_u_conv} and  \ref{Thm_cve_Galerkin_val}, as well as Corollary \ref{Lem_controller_est}, seem thus to be reasonably accessible within this approach.  This way, more realistic models, including e.g.~the coupling of an EBM with a deep ocean such as dealt with in \cite{diaz2014multiple}, could be considered. Similarly, the inclusion of more realistic dynamic boundary conditions for such systems could benefit from the approximate controllability study of \cite{bejenaru2001abstract}.

2. Another promising direction  is the synthesis of (approximate) optimal controls in a feedback form from the Hamilton-Jacobi-Bellman (HJB)  equation associated with Galerkin approximations \cite{KroenerKunischZidani2015,GarckeKroener2016}.  This is particularly relevant for the optimal control of systems near the first criticality in which only very few modes\footnote{Such as a pair of modes in the case of a Hopf bifurcation \cite{ma2005bifurcation}.} have lost their stability and where Galerkin approximations are very often useful to approximate the dynamics near the bifurcated states \cite{dijkstra_al16} although center manifold reduction techniques lead often to further reducing the number of resolved modes needed to approximate accurately the dynamics \cite{ma2005bifurcation,PTD-MW,CLW15_vol1}.  
We refer to \cite{CKL17_DDE} for a first study along this direction, in the context of a Hopf bifurcation arising in a nonlinear delay differential equation (DDE).

3. However, far from the first criticality or when the nonlinear effects get amplified  a larger number of modes is required to dispose of good Galerkin approximations of, already,  the uncontrolled dynamics; see \cite{CL15,CLW15_vol2}.  The numerical burden of the synthesis of controls at a nearly optimal cost---by solving the HJB equation corresponding to the Galerkin approximation---becomes then quickly prohibitive, especially for the case of locally distributed controls; see \cite[Sect.~7]{CL15}.  One avenue to deal with reduced state space of further reduced dimension, is to search for high-mode parameterizations that help reduce the residual energy contained in the unresolved modes, i.e.~to reduce the RHS of \eqref{Est_contr_diff} involving the terms $\| \Pi_N^\perp y(\cdot;u^*)\|_{L^2(0,T; \cH)}$ and $\|  \Pi_N^\perp y (\cdot; u^{*}_{N})\|_{L^2(0,T; \cH)}$ in Corollary \ref{Lem_controller_est}. The theory of parameterizing manifolds (PM) \cite{CL15,CL16post,CLMcW2016} allows for such a reduction leading typically to approximate controls  coming with error estimates that introduce multiplying factors $0\leq Q <1$
in the ``RHS-like'' of \eqref{Est_contr_diff}; see \cite[Theorem 1 \& Corollary 2]{CL15}. The combination of the Galerkin framework introduced here with the PM reduction techniques of \cite{CL15} constitutes thus an idea that is worth pursuing.

4. Finally,  we emphasize that the framework introduced here is not limited to Galerkin approximations built from eigenfunctions of the linear part. This is particularly useful for evolution equations for which such eigenfunctions are not the 
best choice to build Galerkin approximations. As explained in \cite{CGLW15}, systems of nonlinear DDEs constitute such a type of evolution systems. The optimal control of systems of nonlinear DDEs may thus benefit from the framework introduced here and will be pursued elsewhere. The work \cite{CKL17_DDE} illustrates promising results exploiting this idea.
Such an approach may be also relevant to study the possible impact of geoengineering strategies or GHG mitigation policies on large-scale climatic phenomena such as the El Ni\~no--Southern Oscillation (ENSO) and for which DDEs are known to provide good models able to capture some of the essential features of ENSO's irregularity;  see \cite{Tzip_al94,neelin1998enso,Chek_al14_RP,CGN17}.

\section*{Acknowledgments} 
The authors are grateful to the referees for their comments and for having pointed out relevant references for Sect.~\ref{Sec_appl}.
This work has been partially supported by the Office of Naval Research (ONR) Multidisciplinary University Research Initiative (MURI) grant N00014-12-1-0911 and N00014-16-1-2073 (MDC), and by the National Science Foundation grants DMS-1616981 (MDC) and DMS-1616450 (HL). MDC and AK were supported by the project ``Optimal control of partial differential equations using parameterizing manifolds, model reduction, and dynamic programming'' funded by the {\it Foundation Hadamard/Gaspard Monge Program for Optimization and Operations Research (PGMO).}

%%%%%%%%%%%%%%%%%%%%%%%%

\appendix

\section{Existence of optimal controls}\label{Sec_existence_opt_contr}
We recall hereafter some standard sufficient conditions for the existence of optimal controls. Our approach follows 
\cite{HPUU09} that we adapt to our framework. 
Let $\mathcal U:=L^q(0,T;U)$ with $q\geq 1$, with here $U$ denoting a bounded, closed and convex subset of a separable Hilbert space $V$.
In particular,  $\mathcal U$ is also bounded, closed and convex.

Let $Y$ denotes a separable Hilbert space. Let us assume that we can write the state equation in the form
\begin{align}\label{state-equation}
 \mathcal{M}(y,u)=0,
\end{align}
with $ \mathcal{M}\colon Y \times  \mathcal{U} \rightarrow Z$ a continuous mapping taking values in a Hilbert space $Z$. In practice, the mapping $\mathcal{M}$ characterizes the model equation and involves, typically, differential operators.  

The optimal control problem reads as
\begin{align}\label{problem}
 \min_{(y,u) \in Y \times \mathcal U} J(y,u),\quad \text{subject to } \;\;  \mathcal{M}(y,u) =0.
\end{align}

Existence to this optimal control problem can be ensured under conditions grouped into the following assumption. 

\vspace{3ex}
\noindent{\bf Assumption A.1}
 \begin{itemize}
  \item[(i)] The state equation \eqref{state-equation} has a bounded solution operator 
 \bea
 &\mathcal{U} \rightarrow Y, \\
 & u\mapsto y(u).  
 \eea
  \item[(ii)] The mapping
  \bea
\mathcal{M}: &  \; Y \times \mathcal{U} \rightarrow Z,\\ 
&(y,u) \longmapsto \mathcal{M}(y,u),
\eea
 is continuous for the weak topology.
  \item[(iii)] The cost function $J\colon Y\times  \mathcal{U} \rightarrow \R^+$ is weakly lower semicontinuous.
 \end{itemize}
\vspace{1ex}

\begin{rem}
 Assumption {\bf A.1}-(iii) is satisfied if $J$ is convex and continuous.
\end{rem}

We introduce the feasible set
\begin{align}
F_{\text{ad}}:=\{ (y,u) \in Y \times \mathcal U  \; : \; \mathcal{M}(y,u)=0 \}.
\end{align}

\begin{defi}
 The pair $(\bar y,\bar u) \in Y \times \mathcal U$ is called a solution of \eqref{problem} if
 \begin{align}
 J(\bar u , \bar y) \le J(y,u) \quad \text{for all } (y,u) \in F_{\text{ad}}.
 \end{align}
\end{defi}
\begin{thm}
 Under Assumption {\bf A.1} problem \eqref{problem} has a non-empty set of solutions.
\end{thm}
\begin{proof}
 The proof is classical. We present only a brief sketch; e.g.~\cite[Sect.~1.5.2]{HPUU09}. 
 Since $J\ge 0$ and $F_{\text{ad}}$ is nonempty, the infimum
 exists  and there is a minimizing sequence $(y_k,u_k) \subset F_{\text{ad}}$ with $
  J^*:=\lim_{k\rightarrow \infty} J(y_k,u_k)$. 
 Furthermore, $(u_k)$ is bounded as a sequence in $\mathcal{U}$, and by Assumption {\bf A.1}-(i) the sequence $(y_k)$ is bounded. Thus, by reflexivity of $L^2(0,T;\cH)\times Y$ we can select a subsequence that converges weakly to some limit $(\bar u,\bar y)$.
 Since $\mathcal U$ is convex and closed, we have that $\bar u$  lies in $\mathcal U$, and together with Assumption {\bf A.1}-(ii) we have that the set $F_{\text{ad}}$ is sequentially weakly closed; hence, we have $(\bar y, \bar u) \in F_{\text{ad}}$. With Assumption {\bf A.1}-(iii) we obtain the assertion by a classical argument.
\end{proof}

\end{document}